\documentclass[12pt,a4paper,reqno]{amsart}
\pdfoutput=1

\usepackage[utf8]{inputenc}

\title{Connectoids II: existence of normal trees}

\author{Nathan Bowler
			\and
			Florian Reich}
\address{Universit\"at Hamburg, Department of Mathematics, Bundesstrasse 55 (Geomatikum), 20146 Hamburg, Germany}
\email{\{nathan.bowler, florian.reich\}@uni-hamburg.de}


\usepackage{amsmath,amssymb,amsthm}
\usepackage{mathtools}
\usepackage{tikz}
\usetikzlibrary{fadings}
\usetikzlibrary{arrows}

\tikzset{
	vertex/.style={circle, draw, minimum size=1.5em},
	edge/.style={->, > = latex'}
}

\usepackage{comment}
\usepackage{enumitem}
\usepackage{thm-restate}

\usepackage{xcolor}
\usepackage[unicode]{hyperref}
\hypersetup{
	colorlinks,
    linkcolor={red!60!black},
    citecolor={green!60!black},
    urlcolor={blue!60!black},
}
\usepackage[abbrev, msc-links]{amsrefs}
\usepackage[capitalize, nameinlink]{cleveref}
\usepackage[presets={vec-cev}]{letterswitharrows}

\usepackage[utf8]{inputenc}
\usepackage[T1]{fontenc}
\usepackage{lmodern}
\usepackage[babel]{microtype}
\usepackage[english]{babel}

\usepackage{geometry}
\usepackage{letterswitharrows}

\usepackage{enumitem}

\let\polishlcross=\l
\def\l{\ifmmode\ell\else\polishlcross\fi}

\let\emptyset=\varnothing

\let\theta=\vartheta
\let\rho=\varrho
\let\phi=\varphi

\def\NN{\mathbb N}

\def\cC{{\mathcal C}}

\def\cP{{\mathcal P}}

\def\cK{{\mathcal K}}

\newcommand{\Set}[1]{{\left\lbrace {#1} \right\rbrace}}

\def\set#1:#2{\Set{{#1} \colon {#2}}}

\newcommand{\Up}[1]{\lfloor #1 \rfloor}
\newcommand{\OUp}[1]{\mathring{\lfloor #1 \rfloor}}
\newcommand{\Down}[1]{\lceil #1 \rceil}
\newcommand{\ODown}[1]{\mathring{\lceil #1 \rceil}}

\newcommand{\cf}{\textrm{cf}}
\newcommand{\height}{\textrm{height}}

\theoremstyle{plain}
\newtheorem{thm}{Theorem}[section]

\newtheorem{prop}[thm]{Proposition}
\newtheorem{cor}[thm]{Corollary}
\newtheorem{lemma}[thm]{Lemma}

\newtheorem{problem}[thm]{Problem}

\theoremstyle{definition}

\theoremstyle{remark}
\newtheorem{claim}{Claim}

\def\lqedsymbol{\ifmmode$\lrcorner$\else{\unskip\nobreak\hfil
		\penalty50\hskip1em\null\nobreak\hfil$\rule{1.2ex}{1.2ex}$
		\parfillskip=0pt\finalhyphendemerits=0\endgraf}\fi}

\newenvironment{claimproof}[1][\proofname]
{%
	\proof[#1]%
}
{%
	\endproof%
}

\makeatletter
\@addtoreset{claim}{thm}
\makeatother

\begin{document}
	
\begin{abstract}
	In this series, we introduce and investigate the concept of \emph{connectoids}, which captures the connectivity structure of various discrete objects such as undirected graphs, directed graphs, bidirected graphs, hypergraphs and finitary matroids.
	
	In the first paper, we developed a universal end space theory based on connectoids that unifies the existing end spaces of undirected and directed graphs.
	
	In this paper, we establish normal trees of connectoids as a natural generalisation of normal trees of undirected graphs, which are one of the most important tools in infinite graph theory.
	More precisely, we show that the existence of normal trees of connectoids can be characterised in the same way as for normal trees of undirected graphs:
	We extend Jung's famous characterisation via dispersed sets to connectoids, and prove that normal spanning trees exist if they exist in some neighbourhood of each end.
	Furthermore, we show that a connectoid has a normal spanning tree if and only if its groundset can be well-ordered in a certain way, called \emph{countable separation number}.
\end{abstract}

\maketitle

\section{Introduction}
\emph{Connectoids} are an abstract concept representing the connectivity structure of various discrete objects such as undirected graphs, directed graphs, bidirected graphs, hypergraphs and finitary matroids.
Let $S$ be a set and let $\mathcal{F}$ be a set of finite subsets of $S$ such that
\begin{enumerate}[label=(\roman*)]
	\item for every $F, F' \in \mathcal{F}$ with $F \cap F' \neq \emptyset$ the union $F \cup F'$ is also an element of $\mathcal{F}$, and
	\item $\{s\} \in \mathcal{F}$ for every $s \in S$ and $\emptyset \in \mathcal{F}$.
\end{enumerate}
We call a subset $C \subseteq S$ \emph{connected} if for every two elements $x, y \in C$ there is $F \in \mathcal{F}$ with $F \subseteq C$ and $x, y \in F$.
Moreover, we call the tuple $(S, \cC)$ the \emph{connectoid induced by $\mathcal{F}$}, where~$\cC$ is the set of connected subsets of $S$.

Given an undirected, directed or bidirected graph $G$, or a hypergraph $G$ with finite edges, we set $S:= V(G)$.
Let $\mathcal{F}$ be the set of vertex sets of all finite connected subgraphs of $G$, where `connected' refers to connectivity for undirected graphs and hypergraphs, and refers to strong connectivity for directed and bidirected graphs.
Then a subset of $S$ is connected in the connectoid induced by $\mathcal{F}$ if and only if it is the vertex set of a (strongly) connected subgraph of $G$.
Analogously, given a finitary matroid $(E, \mathcal{I})$, we set $S:= E$ and let $\mathcal{F}$ be the set of ground sets of all finite connected restrictions of $(E, \mathcal{I})$. Then the connected sets of the connectoid induced by $\mathcal{F}$ are precisely the ground sets of all connected restrictions of $(E, \mathcal{I})$.

In the first paper of this series \cite{connectoids1}, we developed a universal end space theory based on connectoids.
This theory unifies the existing end spaces of undirected graphs and directed graphs under omitting limit edges.
Furthermore, it provides a notion of ends for a variety of discrete objects that lacked a notion of ends.

The definition of ends is based on the following generalisation of rays.
A connected set~$N$ is called a \emph{necklace} if there exists a family $(H_n)_{n \in \NN}$ of finite connected sets such that $N = \bigcup_{n \in \NN} H_n$ and $H_i \cap H_j \neq \emptyset$ holds if and only if $|i - j| \leq 1$ for every $i, j \in \NN$.
Given a subset $S' \subseteq S$ we call a maximal connected subset of $S'$ a \emph{component} of $S'$ and let $\cK(S')$ be the set of components of $S'$.
We observe that for every finite set $X \subseteq S$ there is a component in $\cK(N \setminus X)$ that contains \emph{almost all}, i.e.\ all but finitely many, elements of $N$, and we refer to it as the \emph{$X$-tail} of $N$.
Two necklaces $N, N'$ are \emph{equivalent} if for every finite set $X \subseteq S$ the $X$-tails of $N$ and $N'$ are contained in the same component in $\cK(S \setminus X)$.
An equivalence class of necklaces under this relation is an \emph{end} of $(S, \cC)$ and we set \emph{$\Omega(S, \cC)$} to be the set of ends of~$(S, \cC)$.

In this paper, we turn our attention to normal trees, which are one of the most important tools in infinite graph theory.
A rooted tree $T$ in an undirected graph $G$ is \emph{normal} if for every connected subgraph $H$ of $G$ and every two $\leq_T$-incomparable elements $u, v \in V(H) \cap V(T)$ there is $w \in V(H) \cap V(T)$ with $w \leq_T u, v$, where $\leq_T$ refers to the tree-order of $T$~\cite{DiestelBook2016}*{Section 1.5}.

In the context of connectoids we define \emph{normal trees} as auxiliary graphs:
A \emph{weak normal tree} $T$ of a connectoid $(S, \cC)$ is a rooted, undirected tree $T$ with $V(T) \subseteq S$ such that
\begin{itemize}
	\item for every connected set $C \in \cC$ and every two $\leq_T$-incomparable elements $u, v \in C \cap V(T)$ there is $w \in C$ with $w \leq_T u, v$, and
	\item for every two $\leq_T$-comparable elements $u \leq_T v$ there is a connected set $C \in \cC$ containing $u$ and $v$ that avoids every element $w <_T u$~\cite{connectoids1}.
\end{itemize}
If additionally for every rooted ray $R$ in $T$ there exists a necklace that contains almost all elements of $V(R)$ we call $T$ a \emph{normal tree} of $(S, \cC)$~\cite{connectoids1}.
While weak normal trees have the same separation properties as normal trees of undirected graphs, the additional constraint of normal trees ensures that they represent ends in the same way as their undirected counterparts:
The ends of a connectoid $(S, \cC)$ are in one-to-one correspondence to the ends of every normal spanning tree of $(S, \cC)$ \cite{connectoids1}*{Theorem~1.2} and rayless normal trees approximate the ends of a connectoid \cite{connectoids1}*{Theorem~1.3}.

Given an undirected graph $G$ and its corresponding connectoid $(V(G), \cC)$, normal trees of~$(V(G),\cC)$ are not necessarily normal trees of $G$~\cite{connectoids1}:
Let $G$ be a comb with vertex set $V(G):= \{s_n, \ell_n: n \in \NN\}$ and edge set $E(G):=\{\{s_n, s_{n + 1}\}, \{s_n, \ell_n\}: n \in \NN\}$. Then the ray $R$ following the sequence $\ell_1, s_1, \ell_2, s_2, \ell_3, s_3, \dots$ is a normal tree of the corresponding connectoid $(V(G), \cC)$. But the ray~$R$ is not a normal tree in $G$ since $\{s_1, \ell_2\} \notin E(G)$.

However, the concepts of normal trees of $G$ and normal trees of $(V(G), \cC)$ coincide in their existence, i.e.\ there exists a normal tree in $G$ containing $U \subseteq V(G)$ if and only if there exists a normal tree in the corresponding connectoid $(V(G), \cC)$ containing $U$~\cite{connectoids1}:
Every normal tree in $G$ is also a normal tree in $(V(G),\cC)$ since its rooted rays form necklaces.
For the converse, let~$T$ be some normal tree of $(V(G), \cC)$.
Then $T$ is a normal tree containing $U$ in $G \cup T$.
Since the existence of normal trees is closed under taking connected subgraphs \cite{pitz2021proof}*{Theorem~1.2}, there exists a normal tree containing $U$ in $G$.

In this paper,
we show that the existence of normal trees of connectoids can be characterised in the same way as the existence of normal trees of undirected graphs.
This establishes normal trees of connectoids as a natural generalisation of normal trees of undirected graphs, and thus as a sensible notion of normal trees for directed and bidirected graphs, as well as hypergraphs and finitary matroids.

First, we extend Jung's famous characterisation via dispersed sets \cite{jung1969wurzelbaume}*{Theorem~6} to connectoids.
Given a connectoid $(S, \cC)$, a subset $S' \subseteq S$ is \emph{dispersed} if every necklace in $(S, \cC)$ has finite intersection with $S'$.
In fact, we show a stronger statement similar to Pitz' result for undirected graphs \cite{pitz2021quickly}*{Theorem~4}:
we prove that fat~$TK_{\aleph_0}$-dispersedness is sufficient for characterising the existence of normal trees.
In this context one can think informally of a fat~$TK_{\aleph_0}$ as a connectoid obtained from the fat~$TK_{\aleph_0}$ by replacing its edges with certain finite connected sets.

\begin{restatable}{thm}{Jung} \label{thm:jung}
	Let $(S, \cC)$ be a connected connectoid and $U \subseteq S$. Then the following are equivalent:
	\begin{enumerate}[label=(\arabic*)]
		\item\label{itm:normal_tree} for every $r \in S$ there is a normal tree rooted at $r$ that contains $U$,
		\item\label{itm:weak_normal_tree} there is a weak normal tree that contains $U$,
		\item\label{itm:quasi_weak_normal_tree} there is a tree $T$ containing $U$ such that for every $C \in \cC$ and every two $\leq_T$-incomparable elements $u, v \in C \cap V(T)$ there is $w \in C$ with $w \leq_T u, v$,
		\item\label{itm:countable_union} $U$ is a countable union of dispersed sets, and
		\item\label{itm:countable_union_ka} $U$ is a countable union of fat~$TK_{\aleph_0}$-dispersed sets.
	\end{enumerate}
\end{restatable}
\noindent
We will deduce from \cref{thm:jung} that the existence of normal trees is closed under taking certain substructures.

Second, we show that a connectoid $(S, \cC)$ has a normal spanning tree if for every end $ \omega \in \Omega(S, \cC)$ there exists a ``neighbourhood'' of $\omega$ that has a normal spanning tree, which generalises a result for undirected graphs~\cite{kurkofka2021approximating}*{Theorem~1}.
For every finite set $X \subseteq S$ let $K(X, \omega)$ be the unique component in $\cK(S \setminus X)$ containing a necklace of $\omega$.
More precisely, we show:

\begin{restatable}{thm}{CharacterisationViaEnds} \label{thm:ends}
	Let $(S, \cC)$ be a connected connectoid and $U \subseteq S$.
	Then there is a normal tree containing $U$ if and only if for every end $\omega \in \Omega(S, \cC)$ there is a finite set $X_\omega \subseteq S$ such that the induced subconnectoid on $K(X_\omega, \omega)$ has a normal tree containing $U \cap K(X_\omega, \omega)$.
\end{restatable}

Third, we transfer the definition of tree-decomposition for directed graphs~\cite{johnson2001directed} to connectoids and generalise a well-known characterisation of the existence of normal trees (see \cite{heuer2017excluding}*{Lemma~2.2} for the forward direction):
\begin{restatable}{lemma}{characterisationTreeDecomp}\label{lem:characterisation_tree_decomp}
	Let $(S, \cC)$ be a connectoid.
	Then there exists a normal spanning tree of $(S, \cC)$ if and only if there exists a tree-decomposition of $(S, \cC)$ into finite parts.
\end{restatable}

Moreover, we prove that the existence of a normal spanning tree of $(S, \cC)$ is equivalent to the existence of a certain well-order of $S$, which is inspired by Pitz' characterisation via countable colouring number in the context of undirected graphs~\cite{pitz2021proof}*{Theorem 1.1}.
A connected connectoid $(S, \cC)$ has \emph{countable separation number} if there exists a well-order $\leq$ of~$S$ such that for every $s \in S$ there exists a finite set $X \subseteq \{t \in S: t < s \}$ with the property that the component in $\cK(S \setminus X)$ containing $s$ avoids $\{t \in S: t < s \}$.

\begin{restatable}{thm}{CountableSeparationNumber}\label{thm:countable_separation_number}
	A connected connectoid $(S, \cC)$ has a normal spanning tree if and only if $(S, \cC)$ has countable separation number.
\end{restatable}

The paper is organised as follows.
We discuss basic properties of necklaces and (weak) normal trees  in \cref{sec:prelims}.
In \cref{sec:jung} we prove \cref{thm:jung} and show \cref{thm:ends} in \cref{sec:ends}.
In \cref{sec:tree_decomposition} we introduce tree-decompositions for connectoids and prove~\cref{lem:characterisation_tree_decomp}.
As a preparation for the proof of \cref{thm:countable_separation_number} we introduce normal partition trees in \cref{sec:normal_substructure} and study countable separation number in \cref{sec:countable_separation_number}.
We show in \cref{sec:decomposition} that connectoids can be decomposed into specific subconnectoids of finite adhesion and prove \cref{thm:countable_separation_number} in \cref{sec:characterisation_countable_separation_number}.
Finally, we present open problems in \cref{sec:open_problems}.

\section{Preliminaries} \label{sec:prelims}
For standard graph-theoretic notations we refer to Diestel's book \cite{DiestelBook2016}.
Let $\NN:= \{1, 2, 3, \dots \}$ and $\NN_0 := \{0\} \cup \NN$.
Given a partial order $\leq$ of a set $A$, a set $B \subseteq A$ is \emph{cofinal} in $A$ if for every $a \in A$ there is $b \in B$ with $a \leq b$.
The \emph{cofinality} of $A$ is defined as $\cf(A):= \inf \{|B|: B \text{ is cofinal in } A\}$.
A cardinal $\kappa$ is \emph{regular}, if $\kappa = \cf(\kappa)$ and \emph{singular} otherwise.
A rooted tree $T$ contains $A \subseteq V(T)$ \emph{cofinally}, if $A$ is cofinal in $V(T)$ with respect to the tree-order of $T$.
Given a partial order $\leq$ of a set $S$ and some element $s \in S$ we define $\Up{s}_\leq:= \{ x \in S: x \geq s \}$ and $\OUp{s}_\leq:= \Up{s}_\leq \backslash \{s\}$.
The sets $\Down{s}_\leq$ and $\ODown{s}_\leq$ are defined analogously.

An \emph{order tree} is a partial order with a unique minimal element whose down-closures are well-ordered.
For simplicity, we write $t \in T$ instead of $t \in V(T)$ for a tree $T$.
Let $T$ be an order tree or a rooted tree.
We simply write $\Up{s}_T, \OUp{s}_T, \Down{s}_T$ and $\ODown{s}_T$ for the sets $\Up{s}_{\leq_T}, \OUp{s}_{\leq_T}, \Down{s}_{\leq_T}$ and $\ODown{s}_{\leq_T}$, respectively, where $\leq_T$ refers to the tree-order of $T$.
We call every $\leq_T$-down-closed set of pairwise $\leq_T$-comparable elements in $V(T)$ a \emph{branch} of $T$.
The \emph{height} of $t \in T$ is the order type of $\ODown{t}_T$.
If the height of $t \in T$ is a limit, we also refer to $t$ as a \emph{limit}.
Furthermore, for every ordinal $\alpha$, the set $T^\alpha:= \{t \in T: t \text{ has height } \alpha \}$ is a \emph{level} of $T$.

Furthermore, we call a sequence $(S_i)_{i < \kappa}$ of sets \emph{continuous} if $S_\alpha = \bigcup_{i < \alpha} S_i$ for every limit $\alpha < \kappa$.

\begin{prop}\label{prop:well-order_respecting}
	Let $S$ be a set and let $\preceq$ be a partial order of $S$ such that $\Down{s}_\prec$ is finite for every $s \in S$.
	Then there is a well-order of $S$ of order type $|S|$ that respects $\preceq$.
\end{prop}
\begin{proof}
	Let $\leq$ be a well-order of $S$ of order type $|S|$.
	We construct an injective function $f: S \rightarrow S$ such that $f(s) \leq f(s')$ for every $s \preceq s' \in S$.
	Then $\leq$ is a well-order on the image of $f$ that respects $\preceq$.
	In particular, its order type is at most $|S|$ and thus precisely $|S|$.
	
	Let $S(i)$ be the initial segment of $(S, \leq)$ of order type $i$ for every $i \leq |S|$.
	We construct a sequence $(f_i)_{i \leq |S|}$ of injective functions $f_i: S \rightarrow S$ such that $f_i(s) \leq f_i(s')$ for every $s \preceq s' \in S$ with $f_i(s)\in S(i)$ and such that for every $s \in S$ and for every limit $\alpha < |S|$ there is $j < \alpha$ such that $(f_i(s))_{j \leq i \leq \alpha}$ is constant.
	
	Set $f_0$ to be the identity.
	We assume that $f_i$ has been defined for some $i < |S|$ and let $x \in S$ such that $\ODown{x}_{\leq}=S(i)$.
	If $x$ is not in the image of $f_i$, set $f_{i+1}:=f_i$.
	Otherwise, let $a \in S$ with $f_i(a)=x$ and let $b \in \Down{a}_{\preceq}$ some $\preceq$-minimal element with the property $f_i(b) \notin S(i)$.
	Note that possibly $a = b$.
	Let $f_{i+1}$ be obtained from $f_i$ by setting $f_{i+1}(a):= f_i(b)$ and $f_{i+1}(b):= f_i(a)$.
	Then $f_{i+1}(s) \leq f_{i+1}(s')$ for every $s \preceq s' \in S$ with $f_i(s)\in S(i+1)$.
	
	Given a limit $\alpha \leq |S|$, we assume that $f_i$ has been defined for every $i < \alpha$.
	Then, for every $s \in S$, there are at most $|\Down{s}_{\preceq}|$ many $i<\alpha$ such that $f_i(s) < f_{i+1}(s)$ and at most one $i < \alpha$ such that $f_i(s) > f_{i+1}(s)$.
	Thus there is $j < \alpha$ such that $(f_i(s))_{j \leq i < \alpha}$ is constant since for every limit $\beta < \alpha$ there is $j' < \beta$ such that $(f_i(s))_{j' \leq i \leq \beta}$ is constant by assumption.
	 We set $f_\alpha(s):= f_j(s)$.
	Note that $f_\alpha$ has the property that $f_\alpha(s) \leq f_\alpha(s')$ for every $s \preceq s' \in S$ with $f_\alpha(s)\in S(\alpha)$.
	Finally, $f:=f_{|S|}$ is as desired.
\end{proof}

\subsection{Connectoids and substructures}
Given a set $S$ and a set $\mathcal{F}$ of finite subsets of $S$ such that
\begin{enumerate}[label=(\roman*)]
	\item\label{itm:bonding_closed} for every $F, F' \in \mathcal{F}$ with $F \cap F' \neq \emptyset$ the union $F \cup F'$ is also an element of $\mathcal{F}$, and
	\item\label{itm:contains_singletons} $\{s\} \in \mathcal{F}$ for every $s \in S$ and $\emptyset \in \mathcal{F}$,
\end{enumerate}
we call a subset $C \subseteq S$ \emph{connected} if for every two elements $x, y \in C$ there is $F \in \mathcal{F}$ with $F \subseteq C$ and $x, y \in F$.
The tuple $(S, \cC)$ is the \emph{connectoid induced by $\mathcal{F}$}, where~$\cC$ is the set of connected subsets of $S$.
\begin{prop}\label{prop:finite_connected_set}
	Let $(S, \cC)$ be a connectoid induced by $\mathcal{F}$.
	Then $\mathcal{F}$ is the set of finite connected sets.
\end{prop}
\begin{proof}
	Every element of $\mathcal{F}$ is connected by definition, and finite.
	Conversely, given a finite connected set $C$, fix $c \in C$ and let $F_{c'} \in \mathcal{F}$ with $F_{c'} \subseteq C$ and $c,c' \in F$ for $ c' \in C$.
	Then $C = \bigcup_{c' \in C} F_{c'}$ is an element of $\mathcal{F}$ by~\labelcref{itm:bonding_closed}.
\end{proof}
\noindent
Note that if $X, Y \in \cC$ intersect, then $X \cup Y \in \cC$.

We call a connectoid $(S, \cC)$ \emph{connected} if $S$ is a connected set.
Given a subset $S' \subseteq S$ we call a maximal connected subset of $S'$ a \emph{component} of $S'$ and let $\cK(S')$ be the set of components of $S'$.
Given some tree $T$ with $V(T) \subseteq S$ and some $t \in T$, we define $K_t^T$ to be the unique component in $\cK(S \setminus\ODown{t}_T)$ containing $t$.

We present three different types of substructures that inherit the structure of a connectoid $(S, \cC)$.
Firstly, we call a connectoid $(S', \cC')$ a \emph{subconnectoid} of $(S, \cC)$ if $S' \subseteq S$ and $\cC' \subseteq \cC$.
\begin{prop}
	Let $(S, \cC)$ be a connectoid and let $S' \subseteq S$.
	Then $(S', \cC \cap \cP(S'))$ is a connectoid.
\end{prop}
\noindent
We call $(S', \cC \cap \cP(S'))$ \emph{the induced subconnectoid} of $(S, \cC)$ on $S'$, or simply an \emph{induced subconnectoid}.
\begin{proof}
	The set $\mathcal{F}'$ of finite elements in $\cC \cap \cP(S')$ satisfies~\labelcref{itm:bonding_closed,itm:contains_singletons}.
	Let $(S',\mathcal{C}')$ be the connectoid induced by $\mathcal{F}'$.
	We show that $\cC \cap \cP(S') = \mathcal{C}'$.
	
	For every $C \in \cC \cap \cP(S')$ and every $x, y \in C$ there is a finite connected set $F$ with $F \subseteq C$ and $x,y \in F$, by~\cref{prop:finite_connected_set}.
	Then $F \in \mathcal{F}'$ by the choice of $\mathcal{F}'$.
	This implies that $C \in \mathcal{C}'$.
	Conversely, let $C' \in \mathcal{C}'$ be arbitrary.
	Then for every $x', y' \in C'$ there is $F'\in \mathcal{F}'$ with $F' \subseteq C'$ and $x',y' \in F'$.
	By the choice of $\mathcal{F}'$ and by~\cref{prop:finite_connected_set}, $C' \in \cC$, and, in particular, $C' \in \cC \cap \cP(S')$.
\end{proof}

Secondly, we consider minors.
\begin{prop}
	Let $(S, \cC)$ be a connectoid and let $\mathcal{P}$ be a partition of $S$ into elements of $\cC$.
	Then $(\mathcal{P}, \{X \subseteq \mathcal{P}: \bigcup X \in \cC \})$ is a connectoid.
\end{prop}
\noindent
Given a connectoid $(S, \cC)$ and a partition $\mathcal{P}$ of $S$ into elements of $\cC$, we say $(\mathcal{P}, \{X \subseteq \mathcal{P}: \bigcup X \in \cC \})$ is obtained from $(S, \cC)$ \emph{by contracting} $\mathcal{P}$.
\begin{proof}
	The set $\mathcal{F}'$ of finite elements in $\{X \subseteq \mathcal{P}: \bigcup X \in \cC \}$ satisfies~\labelcref{itm:bonding_closed,itm:contains_singletons}.
	Let $(\mathcal{P},\mathcal{C}')$ be the connectoid induced by $\mathcal{F}'$.
	We show that $\{X \subseteq \mathcal{P}: \bigcup X \in \cC \} = \mathcal{C}'$.
	
	Let $X \subseteq \mathcal{P}$ with $\bigcup X \in \cC$ be arbitrary.
	Furthermore, let $P_1, P_2 \in X$ be arbitrary.
	We pick some $x \in P_1$ and $y \in P_2$.
	Since $\bigcup X \in \cC$, there is a finite set $F \in \cC$ with $F \subseteq \bigcup X$ and $x, y \in F$, by~\cref{prop:finite_connected_set}.
	Let $F':= \{P \in \mathcal{P}: P \cap F \neq \emptyset\}$, and note that $F'$ is finite since $F$ is finite.
	Then $F \cup \bigcup F' = \bigcup F' \in \cC$ and, in particular, $F'$ is contained in $\mathcal{F}'$.
	Moreover, $P_1, P_2 \in F'$ since $x,y \in F$, and $F' \subseteq X$ since $F \subseteq \bigcup X$.
	This shows that $X \in \mathcal{C}'$.
	
	Conversely, let $C \in \mathcal{C}'$ be arbitrary.
	We show that $\bigcup C \in \cC$, as desired.
	Let $x,y \in \bigcup C$ be arbitrary, and let $Q_1, Q_2 \in C'$ with $x \in Q_1$ and $y \in Q_2$.
	Since $C \in \mathcal{C}'$, there is a finite set $Y \subseteq \mathcal{P}$ with $\bigcup Y \in \cC$, $Y \subseteq C$ and $Q_1,Q_2 \in Y$.
	Then there is a finite set $F \in \cC$ with $F \subseteq \bigcup Y \subseteq \bigcup C$ and $x,y \in F$, by~\cref{prop:finite_connected_set}.
	This implies $\bigcup C \in \cC$.
\end{proof}
\noindent
A connectoid $(S'', \cC'')$ is a \emph{minor} of $(S, \cC)$ if there exists a subconnectoid $(S', \cC')$ of $(S, \cC)$ such that $(S'', \cC'')$ is obtained from $(S', \cC')$ by contracting $S''$.
Note that subgraphs and minors of an undirected graph $G$ correspond to subconnectoids and minors of the connectoid corresponding to $G$.

Thirdly, for every $\hat S \subseteq S$, the tuple $(\hat S, \{C \cap \hat S: C \in \cC \})$ is the \emph{torso} of $(S, \cC)$ at $\hat S$.
\begin{prop}
	Let $(S, \cC)$ be a connectoid and let $\hat S \subseteq S$.
	Then $(\hat S, \{C \cap \hat S: C \in \cC \})$ is a connectoid.
\end{prop}
\begin{proof}
	The set $\mathcal{F}'$ of finite elements in $ \{C \cap \hat S: C \in \cC \}$ satisfies~\labelcref{itm:bonding_closed,itm:contains_singletons}.
	Let $(\hat S,\mathcal{C}')$ be the connectoid induced by $\mathcal{F}'$.
	We show that $ \{C \cap \hat S: C \in \cC \} = \mathcal{C}'$.
	
	Let $C \in \cC$ be arbitrary.
	For every $x, y \in C \cap \hat S$ there is a finite set $F \in \cC$ with $x,y \in F$ and $F \subseteq C$.
	Then $x,y \in F \cap \hat S$, $F \cap \hat S \subseteq C \cap \hat S$ and $F \cap \hat S \in \mathcal{F}'$.
	Thus $C \cap \hat S \in \cC'$.
	
	Conversely, let $C' \in \mathcal{C}'$ be arbitrary.	
	For every $x', y' \in C'$ there is $F_{x',y'} \in \mathcal{F}'$ with $F_{x',y'} \subseteq C'$ and $x',y' \in F_{x',y'}$.
	Let $C_{x',y'}$ be some element of $\cC$ with $C_{x',y'} \cap \hat S = F_{x',y'}$.
	We show that $Z:= \bigcup_{x',y' \in C'} C_{x',y'}$ is in $\cC$, which implies that $C'= Z \cap \hat S$ is in $\{C \cap \hat S: C \in \cC \}$, as desired.
	For $a,b \in Z$ there are $x',y',x'',y'' \in C'$ such that $a \in C_{x',y'}$ and $b \in C_{x'',y''}$.
	Then $C_{x',y'} \cup C_{x',x''} \cup C_{x'',y''} \subseteq Z$ is an element of $\cC$, which implies that there is a finite element $F_{a,b} \in \cC$ with $F_{a,b} \subseteq Z$ and $a,b \in F_{a,b}$.
	Thus $Z$ is indeed in $\cC$.
\end{proof}

\subsection{Necklaces}
A connected set~$N$ is called a \emph{necklace} if there exists a family $(H_n)_{n \in \NN}$ of finite connected sets such that $N = \bigcup_{n \in \NN} H_n$ and $H_i \cap H_j \neq \emptyset$ holds if and only if $|i - j| \leq 1$ for every $i, j \in \NN$.

Given a finite set $X \subseteq S$ and a necklace $N$, we call the component of $\cK(N \setminus X)$ that contains \emph{almost all}, i.e.\ all but finitely many, elements of $N$ the \emph{$X$-tail} of $N$:
\begin{prop}[\cite{connectoids1}*{Corollary~2.3}]\label{prop:tail}
	Let $N \subseteq S$ be a countably infinite connected set in some connectoid $(S, \cC)$.
	Then $N$ is a necklace if and only if there is a component in $\cK(N \setminus X)$ that contains almost all elements of $N$ for every finite set $X \subseteq N$.\\
	Furthermore, the $X$-tail of a necklace $N$ is a necklace for every finite set $X \subseteq S$.
\end{prop}
\noindent
Two necklaces $N, N'$ are \emph{equivalent} if for every finite set $X \subseteq S$ the $X$-tails of $N$ and $N'$ are contained in the same component in $\cK(S \setminus X)$.
An equivalence class of necklaces under this relation is an \emph{end} of $(S, \cC)$ and we set \emph{$\Omega(S, \cC)$} to be the set of ends of~$(S, \cC)$.

For every finite set $X \subseteq S$ let $K(X, \omega)$ be the unique component in $\cK(S \setminus X)$ containing a necklace of $\omega$.
We say an infinite set $Y \subseteq S$ \emph{converges to an end $\omega$} if $K(X, \omega)$ contains almost all elements of $Y$ for every finite set $X \subseteq S$.
Note that an infinite set converges to at most one end.

\begin{prop}[\cite{connectoids1}*{Proposition~2.2}]\label{prop:converging}
		Let $(S, \cC)$ be a connectoid and $X \subseteq S$ a countably infinite set. Then the following properties are equivalent:
	\begin{enumerate}[label=(\roman*)]
		\item\label{itm:converging_1} there exists an end $\omega \in \Omega(S, \cC)$ such that $X$ converges to $\omega$,
		\item\label{itm:converging_2} for every $Y \in S^{< \infty}$ there is a component in $\cK(S \setminus Y)$ that contains almost all elements of $X$, and
		\item\label{itm:converging_3} there exists a necklace that contains almost all elements of $X$.
	\end{enumerate}
	Furthermore, if $(S, \cC)$ is connected, in \labelcref{itm:converging_3} `almost all' can be replaced by `all'.
\end{prop}

\subsection{Weak normal trees}
A \emph{weak normal tree} $T$ of a connectoid $(S, \cC)$ is a rooted, undirected tree $T$ with $V(T) \subseteq S$ such that
\begin{itemize}
	\item for every connected set $C \in \cC$ and every two $\leq_T$-incomparable elements $u, v \in C \cap V(T)$ there is $w \in C \cap \Down{u}_T \cap \Down{v}_T$, and
	\item for every two $\leq_T$-comparable elements $u \leq_T v$ there is a connected set $C \in \cC$ containing $u$ and $v$ that avoids $\ODown{u}_T$~\cite{connectoids1}.
\end{itemize}

\begin{prop}[\cite{connectoids1}*{Proposition~2.4}]\label{prop:equivalence_weak_normal_tree}
	Let $(S, \cC)$ be a connectoid and $T$ a rooted, undirected tree with $V(T) \subseteq S$. Then the following properties are equivalent:
	\begin{enumerate}[label=(\alph*)]
		\item\label{itm:equivalence_weak_normal_tree_1} $T$ is a weak normal tree, and
		\item\label{itm:equivalence_weak_normal_tree_2} $K_t^T \cap V(T) = \Up{t}_T$ holds for every vertex $t \in T$.
	\end{enumerate}
	Furthermore, $K_t^T = \Up{t}_T$ holds for every vertex $t$ of a weak normal spanning tree $T$.
\end{prop}

\begin{prop}\label{lem:nt_properties}
	Let $T$ be a weak normal tree of a connectoid $(S, \cC)$. Then the following properties are true:
	\begin{enumerate}[label=(\roman*)]
		\item\label{itm:nt_property_1}	every rooted subtree of $T$ is a weak normal tree, and
		\item\label{itm:nt_property_2} every connected set $C$ intersecting $V(T)$ has a unique $\leq_T$-least element.
	\end{enumerate}
	Furthermore, if $T$ is a weak normal spanning tree, then the components in $\cK(S \setminus V(T'))$ for every rooted subtree $T' \subseteq T$ are the sets $\Up{t}_T$ for $t$ being $\leq_T$-minimal in $T - T'$.
\end{prop}
\begin{proof}
	The conditions for a weak normal tree are closed under taking rooted subtrees, which proves \labelcref{itm:nt_property_1}.
	
	Let $C$ be some connected set intersecting $V(T)$.
	Let $s \in C \cap V(T)$ such that $\ODown{s}_T \cap C = \emptyset$.
	By \cref{prop:equivalence_weak_normal_tree}, all elements of $C \cap V(T)$ are contained in $\Up{s}_T$, which proves \labelcref{itm:nt_property_2}.
	
	Let $T$ be spanning and $T'$ a rooted subtree of $T$.
	Note that $K_t^T = \Up{t}_T$ holds for every $t \in T$ by \cref{prop:equivalence_weak_normal_tree}.
	Thus for every $\leq_T$-minimal element $t$ in $T - T'$ the set $\Up{t}_T$ is a component in $\cK(S \setminus V(T'))$.
	This finishes the proof.
\end{proof}

Let $T$ be a weak normal tree of a connectoid $(S, \cC)$ and let $K$ be a component in $\cK(S \setminus V(T))$.
We call the set $N_K:= \{t \in T: K_t^T \supseteq K \}$ the \emph{neighbourhood} of $K$ in $T$.

\begin{prop}[\cite{connectoids1}*{Proposition~6.1}]\label{prop:neighbourhood}
	Let $(S, \cC)$ be a connectoid, let $T$ be a weak normal tree of $(S, \cC)$ and let $K$ be a component in $\cK(S \setminus V(T))$.
	Then $K$ is a component in $\cK(S \setminus N_K)$ and $N_K$ is a branch.
\end{prop}

Let $(S, \cC)$ be a connectoid and let $\hat S \subseteq S$ be some subset.
A component $K \in \cK(S \setminus \hat S)$ has \emph{finite adhesion} to $\hat S$ if there exists a finite set $X \subseteq \hat S$ such that $K$ is a component in $\cK(S \setminus X)$.
Otherwise, the component $K$ has \emph{infinite adhesion} to $\hat S$.

\begin{prop}[\cite{connectoids1}*{Proposition~6.2}]\label{prop:finite_neighbourhood}
	Let $(S, \cC)$ be a connected connectoid and $T$ a weak normal tree of $(S, \cC)$.
	Then a component $K \in \cK(S \setminus V(T))$ has finite adhesion to $V(T)$ if and only if $N_K$ is finite.
\end{prop}

\begin{prop}\label{prop:infinite_adhesion_witness}
	Let $(S, \cC)$ be a connectoid and let $\hat S \subseteq S$.
	Then a component $K \in \cK(S \setminus \hat S)$ has infinite adhesion to $\hat S$ if and only if for every $b \in K$ there is a family $(H_n)_{n \in \NN}$ of finite connected sets such that $b \in H_n$, $H_n \cap \hat S \neq \emptyset$ and $H_n \cap H_m \cap \hat S = \emptyset$ for every $n \neq m \in \NN$.
\end{prop}
\begin{proof}
	Let $K \in \cK(S \setminus \hat S)$ be a component that has infinite adhesion to $\hat S$ and let $b \in K$ be arbitrary.
	We construct the desired family $(H_n)_{n \in \NN}$ as follows:
	Let $H_1$ be a finite connected set that contains $b$ and intersects $\hat S$.
	If $(H_m)_{m \leq n}$ has been constructed for some $n \in \NN$, then $K$ is not a component in $\cK(S \setminus (\hat S \cap \bigcup_{m \leq n} H_m))$.
	Thus the component $K' \in \cK(S \setminus (\hat S \cap \bigcup_{m \leq n} H_m))$ that contains $b$ intersects $\hat S$ and we pick a finite connected set $H_{n+1} \subseteq K'$ that contains $b$ and intersects $\hat S$.
	Then $(H_n)_{n \in \NN}$ is as desired.
	
	Conversely, if there exists such a family $(H_n)_{n \in \NN}$, then for every finite set $X \subseteq \hat S$ there is $n \in \NN$ such that $H_n$ avoids $X$.
	Since $K$ avoids $\hat S$ and since $H_n$ intersects both $K$ and $\hat S$, $K$ is not a component in $\cK(S \setminus X)$.
	Thus $K$ has indeed infinite adhesion.
\end{proof}

\subsection{Normal trees}
A weak normal tree $T$ of $(S, \cC)$ is a \emph{normal tree} of $(S, \cC)$ if for every rooted ray $R$ in $T$ there is a necklace that contains almost all elements of $V(R)$.
Note that \cref{prop:converging} characterises this condition.

\begin{prop}[\cite{connectoids1}*{Proposition~6.3}] \label{prop:extension_normal_tree}
	Let $(S, \cC)$ be a connected connectoid and let $T$ be a normal tree of $(S, \cC)$ rooted at $r$.
	Let $T_K$ be a (possibly trivial) normal tree of the induced subconnectoid on $K$ for every component $K \in  \cK(S \setminus V(T))$.
	Suppose that every component in $\cK(S \setminus V(T))$ has finite adhesion to $V(T)$.
	Then there exists a normal tree $T'$ of $(S, \cC)$ rooted at $r$ with $T \subseteq T'$ such that $T' - T$ is the disjoint union of the trees $T_K$ for $K \in \cK(S \setminus V(T))$.
	
	Furthermore, if $T$ and all $T_K$ for $K \in \cK(S \setminus V(T))$ are rayless, then $T'$ is rayless.
\end{prop}

\begin{cor}[\cite{connectoids1}*{Corollary 6.4}]\label{cor:finite_normal_tree}
	Let $(S, \cC)$ be a connectoid, let $s \in S$ be some element and let $X \subseteq S$ be a finite set.
	Then there exists a finite normal tree of $(S, \cC)$ rooted at $s$ with $V(T) = X \cup \{s\}$.
\end{cor}

\subsection{Dispersedness and fat~$TK_{\aleph_0}$-dispersedness}
The union of a ray $R$ with infinitely many disjoint paths that have precisely their first vertex on $R$ is called a \emph{comb}~\cite{DiestelBook2016}.
For an undirected graph $G$ we call a subset $U \subseteq V(G)$ \emph{dispersed} if every comb has finite intersection with $U$~\cite{jung1969wurzelbaume}.
For the definition of dispersed sets in connectoids note that in the connectoid $(V(G), \cC)$ corresponding to an undirected graph $G$ every comb of $G$ forms a necklace in $(V(G), \cC)$.
Vice versa, for every necklace $N$ containing infinitely many elements of $U$ there exists a comb in $G[N]$ that contains infinitely many elements of $U$ by the star-comb lemma~\cite{DiestelBook2016}*{Lemma~8.2.2}, since $\bigcup_{n \in \NN} G[H_n]$ is locally finite for every witness $(H_n)_{n \in \NN}$ of $N$ by the definition of necklace.

Thus we define dispersed sets in general connectoids in the following way:
Let $(S, \cC)$ be a connectoid.
A subset $U\subseteq S$ is called \emph{dispersed} if every necklace in $(S, \cC)$ has finite intersection with $U$.

Given a connectoid $(S, \cC)$ and two disjoint finite sets $X, Y$, we call a finite connected set $C_1 \cup C_2$ an \emph{$X$--$Y$~link} if $C_1$ and $C_2$ are connected sets with $X \subseteq C_1 \setminus C_2$, $Y \subseteq C_2 \setminus C_1$ and $C_1 \cap C_2 \neq \emptyset$.
Two $X$--$Y$~links are \emph{internally disjoint} if they intersect only in $X \cup Y$.
A family of disjoint finite connected sets $(B_n)_{n \in \NN}$ together with uncountably many internally disjoint $B_n$--$B_m$~links for every $n \neq m \in \NN$ is called a \emph{fat $TK_{\aleph_0}$}.
We call the sets in $(B_n)_{n \in \NN}$ the \emph{branch sets} and call a set $U \subseteq S$ \emph{fat~$TK_{\aleph_0}$-dispersed} if for every fat~$TK_{\aleph_0}$ at most finitely many branch sets contain an element of $U$.
Note that we can choose the links of a fat $TK_{\aleph_0}$ such that every $B_n$--$B_m$~link intersects a $B_{n'}$--$B_{m'}$~link only in $(B_n \cup B_m) \cap (B_{n'} \cup B_{m'})$ for every $n \neq m$ and $n' \neq m'$.

\section{Characterisation via dispersed sets}\label{sec:jung}
Jung proved \cite{jung1969wurzelbaume}*{Theorem~6} that an undirected graph $G$ has a normal tree containing $U \subseteq V(G)$ if and only if $U$ is a countable union of dispersed sets.
Pitz expanded \cite{pitz2021quickly}*{Theorem~4} Jung's result by showing that dispersedness can be replaced by fat~$TK_{\aleph_0}$-dispersedness, i.e.\ being finitely separable from every fat~$TK_{\aleph_0}$.
In this section we generalise the results of Jung and Pitz to connectoids.

\Jung*

	\begin{proof}
    \begin{description}
 
		\item[\labelcref{itm:normal_tree} implies \labelcref{itm:weak_normal_tree}] Straightforward.
		\item[\labelcref{itm:weak_normal_tree} implies \labelcref{itm:quasi_weak_normal_tree}] Straightforward.  
        \item[\labelcref{itm:quasi_weak_normal_tree} implies \labelcref{itm:countable_union}] Let $T$ be a tree containing $U$ such that for every $C \in \cC$ and every two $\leq_T$-incomparable elements $u, v \in C \cap V(T)$ there is $w \in C$ with $w \leq_T u, v$. We prove that the distance classes $D_n$ of $T$ with respect to its root are dispersed, which implies \labelcref{itm:countable_union}. More precisely, we show that an arbitrary necklace $N$ has finite intersection with $D_n$ for every $n \in \NN$.
        
        If there exists $t \in T$ such that the $\Down{t}_T$-tail of $N$ is disjoint to $V(T)$, then $N$ has finite intersection with $V(T)$ and in particular, finite intersection with $D_n$ for every $n \in \NN$. Therefore we can assume that the $\Down{t}_T$-tail of $N$ contains a vertex of $V(T)$ for every $t \in T$.
        
        We construct a rooted ray $(t_n)_{n \in \NN}$ in $T$ such that the $\ODown{t_n}_T$-tail of $N$ intersects $V(T)$ only in $\Up{t_n}_T$.
        Then the set $D_n$ has finite intersection with $N$ for every $n \in \NN$ as the $\ODown{t_{n + 1}}$-tail of $N$ intersects $V(T)$ only in $\Up{t_{n + 1}}_T$, which is disjoint to $D_n$.
        
        Let $t_1$ be the root of $T$ and suppose that $(t_n)_{n \leq m}$ has been constructed for some $m \in \NN$.
        The $\Down{t_m}_T$-tail contains a vertex of $T$, by assumption.
        Further, the $\Down{t_m}_T$-tail of $N$ intersects $V(T)$ only in $\OUp{t_m}_T$ since the $\Down{t_m}_T$-tail of $N$ is subset of the $\ODown{t_m}_T$-tail of $N$.
        Let $t_{m+1}$ be a child of $t_m$ such that $\Up{t_{m+1}}_T$ contains an element of the $\Down{t_m}_T$-tail of $N$.
        Then the $\Down{t_m}_T$-tail of $N$ intersects $V(T)$ only in $\Up{t_{m+1}}_T$ since otherwise the $\Down{t_m}_T$-tail of $N$ would contradict the property of $T$.
        As $\ODown{t_{m+1}}_T = \Down{t_m}_T$, this finishes the construction of $(t_n)_{n \in \NN}$.
		
		\item[\labelcref{itm:countable_union} implies \labelcref{itm:countable_union_ka}]
		It suffices to show that each dispersed set $W \subseteq S$ is also fat~$TK_{\aleph_0}$-dispersed.
		Suppose for a contradiction that there exists a fat~$TK_{\aleph_0}$ all whose branch sets contain an element of $W$.
		Let $(B_n)_{n \in \NN}$ be the branch sets of this fat~$TK_{\aleph_0}$.
		Further, let $H_n$ be a $B_n$--$B_{n+1}$~link of this fat~$TK_{\aleph_0}$ avoiding $\bigcup_{i \in [n-2]} H_i$.
		Then $(H_n)_{n \in \NN}$ is a witness of a necklace that contains infinitely many elements of $W$, a contradiction.
		
		\item[\labelcref{itm:countable_union_ka} implies \labelcref{itm:normal_tree}]
			Let $(U_i)_{i \in \NN}$ be a family of fat $TK_{\aleph_0}$-dispersed sets with $U = \bigcup_{i \in \NN} U_i$.
			We construct recursively an increasing sequence of rayless normal trees $(T_n)_{n \in \NN}$ of $(S, \cC)$ with common root and show that the weak normal tree $T:= \bigcup_{n \in \NN} T_n$ is a normal tree containing $U$.
			
			Besides the family $(T_n)_{n \in \NN}$ we construct recursively a function $\Theta: V(T) \rightarrow {S}^{< \infty}$ with $\Theta(u) \subseteq \Theta(v)$ for every $u \leq_T v$.
			For the construction of $\Theta$ fix for every two disjoint finite sets $X, Y \subseteq S$ for which there are at most countably many internally disjoint $X$--$Y$~links a maximal set $\mathcal{P}_{\{X,Y\}}$ of internally disjoint $X$--$Y$~links equipped with an enumeration of order type $\omega$.
			
			Set $T_1:=\{r\}$ and $\Theta(r):=\{r\}$.
			We assume that $T_n$ and $\Theta(v)$ for $v \in V(T_n)$ have been defined for some $n \in \NN$.
			We construct $T_{n + 1}$ by extending $T_n$ finitely into every component in $\cK(S \setminus V(T_n))$ that contains an element of $U$ using \cref{prop:extension_normal_tree}:
			Let $K$ be a component in $\cK(S \setminus V(T_n))$ that contains an element of $U$ and let $i \in \NN$ be minimal with $U_i \cap K \neq \emptyset$.
			Further, pick some $u_K \in U_i \cap K$ and let $t_K \in T_n$ be the $\leq_{T_n}$-maximal element of $N_K$, which is unique by \cref{prop:neighbourhood}.
			We pick a finite normal tree $T_K$ of the induced subconnectoid on $K$ rooted at $u_K$ containing the set $\Theta(t_K) \cap K$, which exists by \cref{cor:finite_normal_tree}.
			Let $T_{n +1}$ be the rayless normal tree obtained by applying \cref{prop:extension_normal_tree} to $T_n$ and $T_K$ for every component $K$ in $\cK(S \setminus V(T_n))$ that contains an element of $U$.
			
			For every $K \in \cK(S \setminus V(T_n))$, $u_K$ is the child of $t_K$ in $T_{n+1}$:
			By construction, $\ODown{u_K}_{T_{n+1}} \subseteq V(T_n)$ and thus $K \subseteq K_{u_K}^{T_{n+1}}$.
			By~\cref{prop:equivalence_weak_normal_tree}, $K_{u_K}^{T_{n+1}} \subseteq K_s^{T_{n+1}}$ and thus $s \in N_K$ for every $s \in \ODown{u_K}_{T_{n+1}}$.
			Furthermore, since $t_K \in N_K$, $u_K \in K_{t_K}^{T_{n}}=K_{t_K}^{T_{n+1}}$ and thus $u_K \in \Up{t_K}_{T_{n+1}}$ by~\cref{prop:equivalence_weak_normal_tree}.
			
			We recursively extend $\Theta$ to $V(T_{n + 1}) \setminus V(T_n)$.
			Let $v$ be some vertex of $V(T_{n+1}) \setminus V(T_n)$ such that $\Theta(v)$ has not been defined but for the parent $u$ of $v$ $\Theta(u)$ has been defined.
			Let $I$ be the set of unordered pairs $\{X, Y\}$ of disjoint finite subsets of $\Theta(u)$ for which there are at most countably many $X$--$Y$~links and such that a link of $\mathcal{P}_{\{X,Y\}}$ is not contained in $\Theta(u)$.
			Let $P_{\{X,Y\}}$ be the minimal link of $\mathcal{P}_{\{X,Y\}}$ that is not contained in $\Theta(u)$.
			Firstly, we define the auxiliary set $\Theta(v)':= \Theta(u) \cup \{v\} \cup \bigcup_{\{X,Y\} \in I} P_{\{X,Y\}}$.
			Secondly, we pick a superset $\Theta(v)$ of $\Theta(v)'$ such that $\Theta(v) \cap K_{x}^{T_{n+1}} $ is connected for every $x \in \Down{v}_{T_{n+1}}$:
			We start with $\Theta(v):=\Theta(v)'$ and apply the following recursion for each element of $\Down{v}_{T_{n+1}}$, which we consider increasingly in $\leq_{T_{n+1}}$-order.
			For $x \in \Down{v}_{T_{n+1}}$ add finitely many vertices to $\Theta(v)$ such that $\Theta(v) \cap K_{x}^{T_{n+1}}$ is connected.
			After the recursion $\Theta(v) \cap K_{x}^{T_{n+1}} $ is still connected for every $x \in \Down{v}_{T_{n+1}}$.
			This finishes the construction of $T$ and $\Theta$.
			
			Given an arbitrary rooted ray $R$ in $T$, we prove three properties of $\Theta(V(R)):= \bigcup_{x \in V(R)} \Theta(x)$ that we use repeatedly in the following.
			We remark that $V(R) \subseteq \Theta(V(R))$ and that every element of $\Theta(V(R))$ is contained in almost all sets $\Theta(x)$ for $x \in V(R)$ since $\Theta(u) \subseteq \Theta(v)$ for every $u \leq_T v$.
			Firstly, $\Theta(V(R)) \cap K_{x}^{T} $ is connected for every $x \in V(R)$ by construction of $\Theta$.
			
			Secondly, for every two disjoint finite subsets $X, Y \subseteq \Theta(V(R))$ there exist uncountably many internally disjoint $X$--$Y$~links if there exists an $X$--$Y$~link that intersects $\Theta(V(R))$ only in $X \cup Y$:
			Otherwise $\bigcup \mathcal{P}_{\{X,Y\}}$ is contained in $\Theta(V(R))$ and the $X$--$Y$~link intersecting $\Theta(V(R))$ only in $X \cup Y$ contradicts the maximality of $\mathcal{P}_{\{X,Y\}}$.
			
			Thirdly, every component $K \in \cK(S \setminus V(T))$ with $N_K = V(R)$ avoids $\Theta(V(R))$:
			Suppose for a contradiction that there exists $\theta \in \Theta(V(R)) \cap K$.
			Pick $n \in \NN$ such that the $\leq_T$-maximal element $y$ of $V(R) \cap V(T_n)$ satisfies $\theta \in \Theta(y)$.
			Let $y^+$ be the child of $y$ in $V(R)$.
			Note that $K_{y^+}^T$ is a component in $\cK(S \setminus V(T_n))$ by \cref{prop:equivalence_weak_normal_tree}.
			Then $\theta \in \Theta(y) \cap K \subseteq \Theta(y) \cap K_{y^+}^T \subseteq V(T_{n+1})$ since $K_{y^+}^T \cap U \neq \emptyset$ by construction of $T$.
			This contradicts $\theta \notin V(T)$ and implies that $\Theta(V(R)) \cap \bigcap_{x \in V(R)} K_x^T = \emptyset$.
			
			By construction, $T$ is a weak normal tree.
			We have to show that $T$ is a normal tree.
			Let $R$ be an arbitrary rooted ray in $T$.
			We construct a witness $(H_n)_{n \in \NN}$ of a necklace and a strictly $\leq_T$-increasing sequence $(x_n)_{n \in \NN}$ of elements of $V(R)$ such that $\Down{x_n}_{T} \subseteq \bigcup_{i \in [n]} H_i \subseteq \Theta(V(R))$ and $K_{x_n}^T \cap \bigcup_{i \in [n-1]} H_i = \emptyset$.
			Then the necklace $\bigcup_{n \in \NN} H_n$ contains all of $V(R)$, which proves that $T$ is a normal tree.
			Set $H_1:=\{r\}$ and $x_1:=r$.
			We assume that $(H_i)_{i \in [n]}$ and $(x_i)_{i \in [n]}$ have been defined.
			Since $\bigcup_{i \in [n]} H_i \subseteq \Theta(V(R)) \subseteq S \setminus \bigcap_{x \in V(R)} K_x^{T}$, there exists $x_{n +1} \in V(R)$  such that $\bigcup_{i \in [n]} H_i \cap K_{x_{n+1}}^T = \emptyset$.
			Let $H_{n+1}$ be a finite connected set in $K_{x_n}^T \cap \Theta(V(R))$ that contains the finite set $\Down{x_{n+1}}_T \setminus \ODown{x_n}_T$.
			Then $(H_n)_{n \in \NN}$ and $(x_n)_{n \in \NN}$ are as desired.
			
			It remains to prove that $T$ contains $U$.
			Suppose for a contradiction that there exists a component $K$ of $\cK(S \setminus V(T))$ that contains an element of $U$.
			If $N_K$ is finite, there exists $n \in \NN$ such that $K$ is a component in $\cK(S \setminus V(T_n))$  by \cref{prop:finite_neighbourhood}.
			But then $T_{n+1}$ contains a vertex of $K$, a contradiction.
			Thus we can assume that $N_K$ is infinite and by \cref{prop:neighbourhood} that $N_K$ is the vertex set of a rooted ray in $T$.
			
			Let $i \in \NN$ such that $U_i \cap K \neq \emptyset$.
			By construction of $T$, $N_K$ contains infinitely many elements of $\bigcup_{j \leq i} U_j$.
			Thus there exists $j \in \NN$ such that $U_j \cap N_K$ is infinite.
			
			Let $k$ be an arbitrary element of $K$.
			We recursively define a family $(C_n)_{n \in \NN}$ of finite connected sets containing $k$ such that $(C_n \cap \Theta(N_K))_{n \in \NN}$ is a family of disjoint connected sets containing an element of $U_j$.
			Then the $(C_n  \cap \Theta(N_K))$--$(C_m  \cap \Theta(N_K))$~link $C_n \cup C_m$ witnesses that there exist uncountably many internally disjoint $(C_n  \cap \Theta(N_K))$--$(C_m  \cap \Theta(N_K))$~links for every $n \neq m \in \NN$.
			Thus there exists a fat~$TK_{\aleph_0}$~whose branch sets $(C_n \cap \Theta(N_K))_{n \in \NN}$ contain elements of $U_j$, contradicting the fat~$TK_{\aleph_0}$-dispersedness of $U_j$.
			
			For the construction of $(C_n)_{n \in \NN}$ assume that $(C_m)_{m \leq n}$ has been defined for some $n \in \NN$.
			Since $\bigcap_{x \in N_K} K_x^T  \cap \Theta(N_K)= \emptyset$, there exist $x \in N_K$ such that $K_x^T \cap \bigcup_{m \leq n} (C_m \cap \Theta(N_K)) = \emptyset$.
			Note that $K_x^T$ contains an element of $U_j \cap \Theta(N_K)$ as $V(R) \setminus \ODown{x}_T \subseteq \Theta(N_K) \cap K_x^T$ contains infinitely many element of $U_j$.
			Since $K_x^T$ is connected there exists a finite connected set $C_{n+1}' \subseteq K_x^T$ that contains $k$ and some element of $U_j \cap \Theta(N_K)$.
			Since $K_x^T \cap \Theta(N_K)$ is connected we can pick a finite connected set $C_{n+1}''$ in $K_x^T \cap \Theta(N_K)$ that contains $C_{n+1}' \cap \Theta(N_K)$.
			Then $C_{n+1}:= C_{n+1}' \cup C_{n+1}''$ is as desired.
			This finishes the construction of $(C_n)_{n \in \NN}$ and thus $T$ is indeed a normal tree rooted at $r$ containing $U$. \qedhere
		\end{description}
	\end{proof}
	By modifying the proof of \cref{thm:jung} slightly we can show:
	\begin{cor}
		Let $(S, \cC)$ be a connected connectoid and $U \subseteq S$. Then there exists a rayless normal tree that contains $U$ if and only if $U$ is dispersed.
	\end{cor}
	\begin{proof}
		For the forward implication we follow the lines of the argument that \labelcref{itm:normal_tree} implies \labelcref{itm:countable_union}.
		Given a rayless normal tree $T$ containing $U$ and some necklace, the construction of the strictly $\leq_T$-increasing sequence terminates and therefore every necklace has finite intersection with $V(T) \supseteq U$.
		
		For the backward implication we use the argument that \labelcref{itm:countable_union_ka} implies \labelcref{itm:weak_normal_tree}.
		Here we construct a normal tree $T$ whose rooted rays contain infinitely many elements of $U$.
		Since $U$ is dispersed, $T$ is rayless.
	\end{proof}
	
\subsection{A refined characterisation for directed graphs}
As undirected and directed graphs form connectoids, \cref{thm:jung} holds for these graphs.
But a fat~$TK_{\aleph_0}$ of an undirected graph is a way more concrete structure than the undirected graph induced by a fat~$TK_{\aleph_0}$ in the corresponding connectoid.
Therefore Pitz' characterisation \cite{pitz2021quickly}*{Theorem~4} of the existence of normal trees in undirected graphs via fat~$TK_{\aleph_0}$-dispersed sets is more precise than \cref{thm:jung}.

In this subsection we elaborate a more precise characterisation of normal trees of directed graphs similar to Pitz' result:
Given a directed graph $D$, a rooted tree $T \subseteq V(D)$ is \emph{normal} if it is normal in the corresponding connectoid $(V(D), \cC)$.
We define a thin subclass of the class of all directed graphs induced by a fat~$TK_{\aleph_0}$ in $(V(D), \cC)$ and show that this class is sufficient for the characterisation of normal trees of directed graphs.

It seems natural to consider the class $\mathbf{C}$ of directed graphs which are subdivisions of the directed multigraph obtained from the complete directed graph $\vec{K}_{\aleph_0}$ by replacing every edge with $\aleph_1$ parallel edges.
But the class $\mathbf{C}$ is not sufficient for a characterisation of normal trees in directed graphs as there exists a directed graph without an element of $\mathbf{C}$ and without normal tree:

Given a directed graph $D'$ we call the directed graph obtained from $D'$ by replacing each vertex $v$ by vertices $v^{\text{in}}$ and $v^{\text{out}}$ such that all incoming edges become incident to $v^{\text{in}}$ and all outgoing edges incident to $v^{\text{out}}$ and by adding a directed edge from $v^{\text{in}}$ to $v^{\text{out}}$ the \emph{split} of $D'$.
Let $H$ be the split of the complete directed graph $\vec{K}_{\aleph_1}$.
Then every vertex of $H$ has either in-degree $1$ or out-degree $1$ and therefore $H$ does not contain an element of $\mathbf{C}$.
Suppose for a contradiction that $H$ contains a normal spanning tree, then by~\cref{thm:jung} $V(H)=\bigcup_{m \in \NN}X_m$ for sets $X_m$ that are dispersed in the corresponding connectoid $(V(H), \cC)$.
Then some $X_m$ contains infinitely many vertices.
Let $(v_n)_{n \in \NN}$ be a sequence of distinct vertices in $\vec{K}_{\aleph_1}$ such that $\{v_n^{\text{in}},v_n^{\text{out}}\} \cap X_m\neq \emptyset$.
Furthermore, let $H_n$ be the connected set induced by the cycle on $v_n^{\text{in}},v_n^{\text{out}}, v_{n+1}^{\text{in}},v_{n+1}^{\text{out}}$.
Then $(H_n)_{n \in \NN}$ is a necklace in $(V(H), \cC)$ that contains infinitely many elements of $X_m$, a contradiction.

We show that the \emph{fat topological split $\vec{K}_{\aleph_0}$}, i.e.\ a subdivision of the split of the directed multigraph obtained from the $\vec{K}_{\aleph_0}$ by replacing every edge with $\aleph_1$ parallel edges, is the appropriate structure for a characterisation of normal trees in directed graphs.
We call a set $W$ \emph{fat topological split $\vec{K}_{\aleph_0}$-dispersed} in a directed graph $D$, if for every fat topological split $\vec{K}_{\aleph_0}$~$K$ there exists a finite set $X \subseteq S$ such that the unique strong component of $ D - X$ containing almost all branch vertices of $K$ avoids $W$.

\begin{thm}
	Let $D$ be a directed graph and $U \subseteq V(D)$.
	Then there exists a normal tree containing $U$ if and only if $U$ is a countable union of fat topological split $\vec{K}_{\aleph_0}$-dispersed sets.
\end{thm}
\begin{proof}
	For the forward direction we use \cref{thm:jung} and show that every dispersed set is also fat topological split $\vec{K}_{\aleph_0}$-dispersed.
	Let $W$ be a dispersed set and suppose for a contradiction that there exists a fat topological split $\vec{K}_{\aleph_0}$~$K$ such that for every finite set $X \subseteq V(D)$ the strong component of $D - X$ containing almost all branch vertices of $K$ contains an element of $W$.
	We call this strong component of $D - X$ the \emph{good} strong component.
	
	We construct recursively a witness $(H_n)_{n \in \NN}$ of a necklace such that $H_n$ contains an element of $W$ from the good strong component of $D - \bigcup_{i \in [n-1]} H_i$.
	If $(H_i)_{i \in [n]}$ has been defined, let $H_{n+1}$ be a finite connected set in the good strong component of $D - \bigcup_{i \in [n-1]} H_i$ that contains an element of $H_n$ and an element of $W$ from the good strong component of $D - \bigcup_{i \in [n]} H_i$.
	Then the necklace $\bigcup_{n \in \NN} H_n$ contains infinitely many elements of $W$, a contradiction.
	
	For the backward direction we follow the lines of the proof that \labelcref{itm:countable_union_ka} implies \labelcref{itm:normal_tree} in \cref{thm:jung}.
	We consider internally disjoint $x$--$y$~paths for vertices $x, y \in V(D)$ instead of internally disjoint links.
	Furthermore, instead of constructing the connected sets $(C_n)_{n \in \NN}$ we construct disjoint finite connected subgraphs $(H_n)_{n \in \NN}$ in $D[\Theta(V(R))]$ such that each $V(H_n)$ contains an element of $U_j$ and distinct vertices $a_n, b_n$ such that there exist a $\Theta(V(R))$--$k$~path starting in $a_n$ and a $k$--$\Theta(V(R))$~path ending in $b_n$.
	By picking a $b_n$--$a_n$~path in $H_n$ for every $n \in \NN$ we obtain a fat topological split $\vec{K}_{\aleph_0}$ that cannot be finitely separated from $U_j$.
\end{proof}

\subsection{Closure under taking subconnectoids and minors}\label{sec:subconnectoid_minor}
In this subsection we deduce from \cref{thm:jung} that the existence of normal trees is closed under taking subconnectoids and minors.
More precisely, we show that dispersedness is closed under these relations.
Let $(S, \cC)$ be a connectoid and $U \subseteq S$ a dispersed set with respect to $(S, \cC)$.

As every necklace in a subconnectoid $(S', \cC')$ of $(S, \cC)$ is also a necklace in $(S, \cC)$, the set $U \cap S'$ is dispersed in $(S', \cC')$.
\begin{cor}\label{cor:normal_tree_closed_subgraph}
    Let $(S, \cC)$ be a connectoid and $U \subseteq S$ a set that is contained in a normal tree of $(S, \cC)$.
    Then for every connected subconnectoid $(S', \cC')$ of $(S, \cC)$ there is a normal tree of $(S', \cC')$ that contains $U \cap S'$.
\end{cor}
Now we consider a minor $(S'', \cC'')$ of $(S, \cC)$ and prove that $\{s \in S'': s \cap U \neq \emptyset\}$ is a dispersed set with respect to $(S'', \cC'')$.
Suppose for a contradiction that there is a necklace $N$ of $(S'', \cC'')$ that has infinite intersection with $\{s \in S'': s \cap U \neq \emptyset\}$, i.e.\ infinitely many elements of $N$ contain an element of $U$.

We construct a witness $(H_n)_{n \in \NN}$ of $N$ such that $\bigcup H_n \cap \bigcup H_{n + 1} \cap U \neq \emptyset$ for every $n \in \NN$:
Let $(\tilde{H}_n)_{n \in \NN}$ be some witness of $N$ and let $(k(n))_{n \in \NN}$ be an increasing sequence of natural numbers with $\bigcup \tilde{H}_{k(n)} \cap U \neq \emptyset$ for every $n \in \NN$.
Set $H_n:= \bigcup_{i \in [k(n), k(n+1)]} \tilde{H}_i$.
Then $(H_n)_{n \in \NN}$ is the desired witness of $N$.

Pick $u_n \in \bigcup H_n \cap \bigcup H_{n + 1} \cap U$ for every $n \in \NN$.
There is a finite connected set $H_n' \subseteq \bigcup H_n$ containing $u_n$ and $u_{n + 1}$ since $\bigcup H_n \in \cC$. Thus $\bigcup_{n \in \NN} H_n'$ is a necklace in $(S, \cC)$ containing infinitely many elements of $U$, contradicting that $U$ is dispersed.
We deduce:
\begin{cor}\label{cor:normal_tree_closed_minor}
    Let $(S, \cC)$ be a connectoid and $U \subseteq S$ a set that is contained in a normal tree of $(S, \cC)$.
    Then for every connected minor $(S'', \cC'')$ of $(S, \cC)$ there is a normal tree of $(S'', \cC'')$ that contains $\{s \in S'': s \cap U \neq \emptyset\}$.
\end{cor}

\section{Characterisation via ends} \label{sec:ends}
We characterise the existence of normal spanning trees via the existence of normal spanning trees in the ``vicinity'' of ends:
\CharacterisationViaEnds*
The main ingredient for this proof is the approximation of connectoids by rayless normal trees:

\begin{thm}[\cite{connectoids1}*{Theorem~1.3}]\label{thm:approx}
Let $(S, \cC)$ be a connected connectoid. For every collection $\mathcal{K} = \{K(X_\omega, \omega): \omega \in \Omega(S, \cC)\}$, where $X_\omega \subseteq S$ is finite for every $\omega \in \Omega(S, \cC)$, there is a rayless normal tree $T$
of $(S, \cC)$ such that every component in $\cK(S \setminus V(T))$ is subset of an element of $\mathcal{K}$.
\end{thm}

\begin{proof}[Proof of \cref{thm:ends}]
    The forward direction follows from \cref{cor:normal_tree_closed_subgraph}.
    For the backward direction, let $\cK:= \{K(X_\omega, \omega): \omega \in \Omega(S, \cC)\}$ be a collection such that the set $U \cap K(X_\omega, \omega)$ is contained in a normal tree of the induced subconnectoid on $K(X_\omega, \omega)$ for every $\omega \in \Omega(S, \cC)$. We apply \cref{thm:approx} to the collection $\cK$ to obtain a rayless normal tree $T$ of $(S, \cC)$ such that every component in $\cK(S \setminus V(T))$ is contained in an element of $\cK$.
    
    Let $K$ be an arbitrary component in $\cK(S \setminus V(T))$ and $\omega \in \Omega(S, \cC)$ such that $K \subseteq K(X_\omega, \omega)$.
    By assumption and by \cref{cor:normal_tree_closed_subgraph}, there is a normal tree $T_K$ of the induced subconnectoid on $K$ that contains $U \cap K$.
    Furthermore, by \cref{prop:neighbourhood} and since $T$ is rayless, $N_K$ is finite and thus $K$ has finite adhesion to $V(T)$ by \cref{prop:finite_neighbourhood}.
    
    We apply \cref{prop:extension_normal_tree} to $T$ and $T_K$ for every $K \in \cK(S \setminus V(T))$ to obtain a normal tree $T'$ in $(S, \cC)$ that contains $U$.
\end{proof}

\section{Characterisation via tree-decomposition into finite parts}\label{sec:tree_decomposition}
For undirected graphs it is well-known that there exists a normal spanning tree if and only if there exists a tree-decomposition into finite parts (see \cite{heuer2017excluding}*{Lemma~2.2} for the forward direction).
In this section we introduce a notion of tree-decomposition for connectoids and show that there exists a normal spanning tree of a connectoid if and only if there exists a tree-decomposition into finite parts.

More precisely, we adapt the common notion of tree-decomposition for directed graphs \cite{johnson2001directed} to connectoids.
A \emph{tree-decomposition} of a connectoid $(S, \cC)$ is a triple $(T, \beta, \gamma)$, where $T$ is a rooted tree, $\beta: V(T) \rightarrow 2^S$, $\gamma: E(T) \rightarrow 2^S$ such that
\begin{itemize}
	\item $(\beta(t))_{t \in V(T)}$ is a partition of $S$,
	\item $\bigcup_{t \in \Up{v}_T} \beta(t)$ is a component in $\cK(S \setminus \gamma(e))$ for every $uv \in E(T)$ with $u <_T v$.
\end{itemize} 
For every $t \in V(T)$ we define $\Gamma(t):= \beta(t) \cup \bigcup_{e \sim t} \gamma(e)$, where $e \sim t$ if $e$ is incident with $t$.
We say $(T, \beta, \gamma)$ is a \emph{tree-decomposition into finite parts} if $\Gamma(t)$ is finite for every $t \in V(T)$.

\characterisationTreeDecomp*
\begin{proof}
	For the forward direction let $T$ be a normal spanning tree of $(S, \cC)$.
	We set $\beta(t):= \{t\}$ for every $t \in V(T)$ and $\gamma(uv):= \Down{u}_T$ for every $uv \in E(T)$ with $u <_T v$.
	By \cref{prop:equivalence_weak_normal_tree}, given $uv \in E(T)$ with $u <_T v$, $\Up{v}_T = \bigcup_{t \in \Up{v}_T} \beta(t)$ is a component in $\cK(S \setminus \Down{u}_T)$.
	Note that $\Gamma(t)= \Down{t}_T$ is finite for every $t \in V(T)$ and thus $(T, \beta, \gamma)$ is a tree-decomposition into finite parts.
	
	For the backward direction let $(T, \beta, \gamma)$ be a tree-decomposition of $(S, \cC)$ into finite parts.
	Let $D_n$ be the $n$-th distance class of $T$ with respect to its root.
	We show that $S_n := \bigcup_{t \in D_n} \beta(t)$ is a dispersed set for every $n \in \NN$.
	Then $S$ is a countable union of dispersed sets and by \cref{thm:jung} there exists a normal spanning tree of $(S, \cC)$.
	
	Let $N$ be an arbitrary necklace in $(S, \cC)$.
	We construct a rooted ray $(t_n)_{n \in \NN}$ of $T$ such that for every $n \in \NN$ some tail of $N$ is contained in $\bigcup_{t \in \Up{t_n}_T} \beta(t)$.
	Set $t_0$ to be the root of $T$ and assume that some $t_n$ has been defined.
	Then some tail of $N$ is contained in $\bigcup_{t \in \OUp{t_n}_T} \beta(t)$ since $\beta(t_n)$ is finite.
	By definition of tree-decomposition, each component in $\cK(S \setminus \Gamma(t_n))$ that intersects $\bigcup_{t \in \OUp{t_n}_T} \beta(t)$ is subset of a set of the form $\bigcup_{t \in \Up{s}_T} \beta(t)$ for a child $s$ of $t_n$.
	Let $t_{n+1}$ be the child of $t_n$ such that the component in $\cK(S \setminus \Gamma(t_n))$ containing the $\Gamma(t_n)$-tail of $N$ is contained in $\bigcup_{t \in \Up{t_{n+1}}_T} \beta(t)$.
	This finishes the construction of $(t_n)_{n \in \NN}$.
	Then $N$ has finite intersection with $S_n$ since some tail of $N$ is contained in $\bigcup_{t \in \Up{t_{n+1}}_T} \beta(t)$, which proves that $S_n$ is dispersed.
\end{proof}

\section{Normal partition trees} \label{sec:normal_substructure}
In this section we introduce normal partition trees for connectoids, which generalise normal partition trees of undirected graphs (see \cite{brochet1994normal} for more details). We prove that every connected connectoid has a normal partition tree.

We begin by extending the concept of normality from trees to order trees.
Let $(S, \cC)$ be a connected connectoid.
An order tree $T$ is a \emph{normal spanning order tree} of $(S, \cC$) if $V(T) = S$ and
\begin{itemize}
	\item for every connected set $C \in \cC$ and every two $\leq_T$-incomparable elements $u, v \in C \cap V(T)$ the set $\Down{u}_T \cap \Down{v}_T \cap C$ is nonempty, and
	\item for every two $\leq_T$-comparable elements $u \leq_T v$ there is a connected set $C \in \cC$ containing $u$ and $v$ such that $C \cap \ODown{u}_T = \emptyset$.
\end{itemize}
\noindent
In the special case of trees a tree $T$ is a normal spanning order tree if and only if $T$ is a normal spanning tree.

\begin{prop}\label{lem:tdigraph_property}
	Let $T$ be a normal spanning order tree of some connectoid $(S, \cC)$. Then the following properties hold:
	\begin{itemize}
		\item every connected set has a unique $\leq_T$-least element, and
		\item for every down-closed subtree $T' \subseteq T$ the components in $\cK(S \setminus V(T'))$ are the sets $\Up{t}_T$ for $t$ $\leq_T$-minimal in $S \setminus V(T')$.
	\end{itemize}
\end{prop}
\begin{proof}
	The same proof as the proof of \cref{lem:nt_properties} since this proof does not use the fact that $T$ is a tree.
\end{proof}

Given a normal spanning order tree $T$ of $(S, \cC)$ we define for every $t \in T$
\[
N_t := \Set{s \in \ODown{t}_T: s, t \text{ are contained in the same component in } \cK(\Up{s}_T \setminus \OUp{t}_T)}.
\]
Now we generalise the concept of $T$-graphs to connectoids (see~\cite{diestel2001normal} for an introduction to $T$-graphs):
The connectoid $(S, \cC)$ is a \emph{$T$-connectoid} if $T$ is a normal spanning order tree of $(S, \cC)$ such that the set $N_t$ is cofinal in $\ODown{t}_T$ for every $t \in T$.

We recall that, given an order tree $T$ and an ordinal $\alpha$, the levels of $T$ are defined as $T^\alpha:= \{t \in T: \ODown{t}_T \text{ has order type } \alpha \}$.

\begin{prop}\label{lem:interval}
	Let $(S, \cC)$ be a $T$-connectoid for some normal spanning order tree $T$. Then every two elements $t <_T t' \in T$ are contained in the same component in $\cK(\Up{t}_T \setminus \OUp{t'}_T)$.
\end{prop}
\begin{proof}
	Let $t \in T$ be arbitrary and let $\alpha$ be the ordinal with $t \in T^{\alpha}$.
	We prove by transfinite induction on $\beta > \alpha$ that the statement is true for every $t' \in T^{\beta} \cap \Up{t}_T$.
	
	We show the induction step for some $\beta > \alpha$ and an arbitrary $t' \in T^{\beta} \cap \Up{t}_T$.
	Since $(S, \cC)$ is a $T$-connectoid, $N_{t'}$ is cofinal in $\ODown{t'}_T$.
	If $\beta$ is a successor, the parent $t^*$ of $t'$ in $T$ is contained in $N_{t'}$.
	If $\beta$ is a limit, there is $t^*$ with $t \leq_T t^* <_T t'$ such that $t^* \in N_{t'}$.

    In both cases, $t^*$ and $t'$ are contained in the same component in $\cK(\Up{t^*}_T \setminus \OUp{t'}_T)$.
    By induction hypothesis, $t^*$ and $t$ are contained in the same component in $\cK(\Up{t}_T \setminus \OUp{t^*}_T)$.
    Then $t'$ and $t$ are contained in the same component in $\cK(\Up{t}_T \setminus \OUp{t'})$ as desired.
\end{proof}
\noindent
Note that \cref{lem:interval} implies that $N_t = \ODown{t}_T$ for every $t \in T$.

\begin{prop}\label{prop:construct_star}
	Let $T$ be an order tree and let $(S, \cC)$ be some $T$-connectoid.
	For every limit $t \in T$, there exists a family $(H_n)_{n \in \NN}$ of finite connected sets with $t \in H_n$, $H_n \cap \ODown{t}_T \neq \emptyset$ and $H_n \cap H_m \subseteq \bigcap_{u \in \ODown{t}_T} \Up{u}_T$ for every $n \neq m \in \NN$.
\end{prop}

\begin{proof}
	Let $t$ be some limit in $T$.
	We construct the desired family $(H_n)_{n \in \NN}$ recursively.
	Let $m \in \NN$ such that $(H_n)_{n \leq m}$ has been constructed.
	Since each element of $(H_n)_{n \leq m}$ is finite, there exists $s \in \ODown{t}_T$ such that $ \Up{s}_T \cap \bigcup_{n \leq m} H_n \subseteq \bigcap_{u \in \ODown{t}_T} \Up{u}_T$.
	
	By \cref{lem:interval}, there exist a finite connected set $H_{m +1} \subseteq \Up{s}_T$ containing $s$ and $t$.
	Then $H_n \cap H_{m +1} \subseteq \bigcap_{u \in \ODown{t}_T} \Up{u}_T$ for every $n \leq m$.
	This finishes the construction of $(H_n)_{n \in \NN}$.
\end{proof}

Now we turn our attention to the definition of normal partition trees.
Let $(S, \cC)$ be a connectoid and let $\mathcal{P} := \{P_t: t \in T \}$ be a partition of $S$ into elements of $\cC$.
If the tuple $(T, \leq_T)$ is an order tree, we call $T$ a \emph{partition tree} of $(S, \cC)$.
Further, we consider the minor $(\mathcal{P}, \cC_{\mathcal{P}})$ obtained by contracting the sets in $\mathcal{P}$ and call the partition tree $T$ \emph{normal} if
\begin{itemize}
	\item $(\mathcal{P}, \cC_{\mathcal{P}})$ is a $T$-connectoid,
	\item for each $t \in T$ we either have $|P_t| \leq \cf(\height(t))$ or $|P_t| < \omega$.
\end{itemize}
\noindent
Note that the properties of \cref{lem:interval} apply to $(\mathcal{P},  \cC_{\mathcal{P}})$ and thus extend to normal partition trees:
\begin{cor}
	Let $T$ be a normal partition tree of a connectoid $(S, \cC)$ with partition $\{P_t: t \in T \}$.
	Then for every two elements $t <_T t'\in T$, the sets $P_t$ and $P_{t'}$ are contained in the same component in $\cK(\bigcup_{s \in \Up{t}_T \setminus \OUp{t'}_T}P_s)$. 
\end{cor}
\begin{lemma}\label{thm:existence_normal_partition_tree}
	Every connected connectoid has a normal partition tree.
\end{lemma}
\noindent
The proof of \cref{thm:existence_normal_partition_tree} follows the lines of the proof of its graph-theoretic counterpart \cite{brochet1994normal}*{Theorem~4.2}.
We use the following basic property:
\begin{prop}\label{prop:connected set_closure}
	Let $(S, \cC)$ be a connected connectoid, let $C$ be a connected set and let $X \subseteq C$.
	Then there is a connected set $C' \subseteq C$ containing $X$ with $|X| = |C'|$ or $|C'| < \omega$.
\end{prop}
\begin{proof}
	Fix some $x \in X$ and consider for every $y \in X$ a finite connected set $C_y \subseteq C$ containing $x$ and $y$. Then $C':= \bigcup_{y \in X} C_y$ is a connected set.
	By construction, either $C'$ is finite or $|C'| = |X|$.
\end{proof}

Let $(S, \cC)$ be a connectoid and let $T$ be a partition tree with a partition $\{P_t: t \in T\}$ of $S$ into elements of $\cC$.
Given a set $U \subseteq V(T)$, we set $P(U):= \bigcup_{u \in U} P_u$ and define $P(T'):= P(V(T'))$ for every subtree $T'$ of $T$.
For $t \in T$, let $K_t$ be the unique component in $\cK(S \setminus P(\ODown{t}_{T}))$ containing $P_t$.

\begin{proof}[Proof of \cref{thm:existence_normal_partition_tree}]
	Let $(S, \cC)$ be a connected connectoid and fix a well-order of $S = \Set{s_\alpha: \alpha < \sigma}$. 
	We will construct a continuous increasing sequence of rooted order trees $T_\alpha$ and a family $(P_t)_{t \in \bigcup_{\alpha \leq \sigma} T_\alpha}$ of disjoint, nonempty connected sets satisfying for every $\alpha \leq \sigma$ the following properties:
	\begin{enumerate}[label={(\alph*)}]
		\item\label{itm:npt1} for every $t, t' \in T_\alpha$, every connected set containing elements of $P_t$ and $P_{t'}$  contains an element of $P(\Down{t}_{T_\alpha} \cap \Down{t'}_{T_\alpha})$,
		\item\label{itm:npt2} for every $t \in T_\alpha$ the set of $s \in \ODown{t}_{T_\alpha}$ that have the property that $P_s$ and $P_t$ are contained in the same component in $\cK((K_s \setminus K_t) \cup P_t)$ is cofinal in $\ODown{t}_{T_\alpha}$,
		\item\label{itm:npt3} for every $t \in T_\alpha$ we either have $|P_t| \leq \cf(\height(t))$ or $|P_t| < \omega$, and
		\item\label{itm:npt4} $\{s_\beta: \beta < \alpha \} \subseteq P(T_\alpha)$.
	\end{enumerate}
    
    We begin by showing that $T:= T_\sigma$ together with $\mathcal{P}:= (P_t)_{t \in T}$ is the desired normal partition tree of $(S, \cC)$.
    By \labelcref{itm:npt4}, $\mathcal{P}$ is a partition of $S$ and therefore $T$ is a partition tree. Further, \labelcref{itm:npt3} ensures that the second condition of normal partition tree is true.
    For the first condition of normal partition trees we have to prove that the minor $(\mathcal{P}, \cC_{\mathcal{P}})$ of $(S, \cC)$ obtained by contracting $\mathcal{P}$ is a $T$-connectoid.
    
    Similarly as in the proof of \cref{lem:interval} one can show by transfinite induction that for every two elements $t \leq_T t' \in T$ the sets $P_t$ and $P_{t'}$ are contained in the same component in $\cK((K_t \setminus K_{t'} ) \cup P_{t'})$ by $\labelcref{itm:npt2}$.
    Then together with \labelcref{itm:npt1} we can deduce that $T$ is a normal spanning order tree of the minor $(\mathcal{P}, \cC_{\mathcal{P}})$.
    Furthermore, $N_t$ is cofinal in $\ODown{t}_T$ for every $t \in T$ by \labelcref{itm:npt2}.
    Thus $(\mathcal{P}, \cC_{\mathcal{P}})$ is indeed a $T$-connectoid.

    It remains to construct the sequence $(T_\alpha)_{\alpha \leq \sigma}$ and the family $(P_t)_{t \in T}$.
	Set $T_0:= \emptyset$.
    For a limit $\alpha \leq \sigma$ we set $T_\alpha:= \bigcup_{\beta < \alpha} T_{\beta}$. Note that if every $T_\beta$ with $\beta < \alpha$ satisfies \labelcref{itm:npt1,itm:npt2,itm:npt3,itm:npt4}, then also $T_\alpha$ satisfies \labelcref{itm:npt1,itm:npt2,itm:npt3,itm:npt4}.
    For the successor step we consider $\alpha \leq \sigma$ such that $T_\alpha$ satisfies \labelcref{itm:npt1,itm:npt2,itm:npt3,itm:npt4}.
    If $s_\alpha \in P(T_\alpha)$, then $T_{\alpha + 1}:= T_\alpha$ satisfies \labelcref{itm:npt1,itm:npt2,itm:npt3,itm:npt4}.
    
    Otherwise, let $K$ be the component in $\cK(S \setminus P(T_\alpha))$ containing $s_\alpha$.
    We consider the set $M_K := \Set{ s \in  V(T_\alpha) : K \subseteq K_s }$.
    Clearly, the set $M_K$ is $\leq_{T_\alpha}$-down-closed.
    Further, $M_K$ is a chain in $T_\alpha$:
    Suppose for a contradiction that there are $\leq_{T_\alpha}$-incomparable elements $s, s'$ such that $K \subseteq K_s, K_{s'}$.
    Then $K_s \cap K_{s'} \neq \emptyset$ and therefore there is a connected set $C \subseteq K_s \cup K_{s'}$ containing $P_s$ and $P_{s'}$. This contradicts \labelcref{itm:npt1} since $(K_s \cup K_{s'}) \cap P(\ODown{s}_{T_\alpha} \cap \ODown{s'}_{T_\alpha}) = \emptyset$.
    
	If $M_K$ has a maximal element $m$, let $T_{\alpha + 1} $ be the order tree obtained from $T_\alpha$ by placing a vertex $t$ on top of $m$. Pick a finite connected set $H \subseteq K_m$ that contains $s_\alpha$ and an element of $P_m$. Let $P_{t} \subseteq K$ be a finite connected set containing $H \cap K$, which exists by \cref{prop:connected set_closure}.
	
	If $M_K$ has no maximal element, let $T_{\alpha + 1} $ be the order tree obtained from $T_\alpha$ by placing a vertex $t$ on top of the chain $M_K$. Take a cofinal subset $M_K' \subseteq M_K$ such that $|M_K'| = \cf(\height(t))$.
    Pick for each element $m \in M_K'$ a finite connected set $H_m \subseteq K_m$ that contains $s_\alpha$ and an element of $P_m$.
    Let $P_t$ be a connected set in $K$ containing $K \cap \bigcup_{m \in M_K'} H_m$ of size at most $\cf(\height(t))$, which exists by \cref{prop:connected set_closure}.
	
	By construction, $T_{\alpha + 1}$ satisfies \labelcref{itm:npt3,itm:npt4}. Next, we prove $K = K_t$. Note that $P(\ODown{t}_{T_{\alpha+1}})$ is a subset of $P(T_\alpha)$. Since $K_t$ is a component in $\cK(S \setminus P(\ODown{t}_{T_{\alpha+1}}))$ and $K$ a component in $\cK(S \setminus P(T_{\alpha+1}))$, $K \subseteq K_t$ holds.
    Suppose for a contradiction that $K_t \setminus K$ is nonempty.
    Then $K_t$ contains an element of some $P_s$ with $s \in V(T_\alpha) \setminus \ODown{t}_{T_{\alpha + 1}}$.
	Let $s \in V(T_\alpha) \setminus \ODown{t}_{T_{\alpha + 1}}$ be $\leq_T$-minimal with this property.
    Then $K_t$ avoids $P(\ODown{s}_{T_\alpha})$ and contains an element of $P_s$.
    Thus $K_t \subseteq K_s$, which implies that $s \in M_K$.
    This gives a contradiction since $s \notin \ODown{t}_{T_{\alpha+1}} = M_K$.
	
	This implies that every connected set containing an element of $P_t$ and an element of $P(T_\alpha)$ intersects $P(\ODown{t}_{T_{\alpha + 1}})$. 
    We can deduce that \labelcref{itm:npt1} holds for $T_{\alpha + 1}$ since \labelcref{itm:npt1} holds for $T_{\alpha}$.
	
	By construction, all elements of $M_K'$ have the property that $P_s, P_t$ are contained in a common component in $\cK((K_s \setminus K) \cup P_t) = \cK((K_s \setminus K_t) \cup P_t)$ witnessed either by the connected set $H$ or by the connected sets $H_m$ for $m \in M_K'$.
	Thus a cofinal subset of $\ODown{t}_{T_{\alpha + 1}}$ has this property, which implies that \labelcref{itm:npt2} holds for $T_{\alpha + 1}$ and completes the proof.
\end{proof}

\section{Countable separation number} \label{sec:countable_separation_number}
In this section we discuss basic properties of connectoids with countable separation number and introduce two classes of connectoids without countable separation number.

We recall the definition of countable separation number.
A connected connectoid $(S, \cC)$ has \emph{countable separation number} if there exists a well-order $\leq$ of $S$ such that for every $s \in S$ there is a finite set $X \subseteq \ODown{s}_\leq$ with the property that the component in $\cK(S \setminus X)$ containing $s$ avoids $\ODown{s}_\leq$.
A partial-order $\leq$ is \emph{well-founded} if there is no infinite $\leq$-descending sequence.
We present the following equivalence:
\begin{lemma}\label{lem:countable_separtion_number_equivalence}
	Let $(S, \cC)$ be a connected connectoid. The following properties are equivalent
	\begin{enumerate}[label=(\arabic*)]
		\item\label{itm:csn_equiv_1} $(S, \cC)$ has countable separation number as witnessed by a well-order of order type $|S|$,
		\item\label{itm:csn_equiv_2} $(S, \cC)$ has countable separation number, and
		\item\label{itm:csn_equiv_3} there exists a well-founded partial order $\preceq$ of $S$ such that for every $s \in S$
		\begin{enumerate}[label=(\roman*)]
			\item\label{itm:csn_equiv_3i} $\Down{s}_\preceq$ is finite,
			\item\label{itm:csn_equiv_3ii} the component in $\cK(S \setminus \ODown{s}_\preceq)$ containing $s$ is subset of $\Up{s}_\preceq$.
		\end{enumerate}
	\end{enumerate}
\end{lemma}
\begin{proof}
	\begin{description}
		\item[\labelcref{itm:csn_equiv_1} implies \labelcref{itm:csn_equiv_2}] Straightforward.
		\item[\labelcref{itm:csn_equiv_2} implies \labelcref{itm:csn_equiv_3}]
		Let $\leq$ be some well-order of $S$ witnessing countable separation number. Then for every $s \in S$ there exists a finite set $X_s \subseteq\ODown{s}_\leq$ such that the component in $\cK(S \setminus X_s)$ containing $s$ avoids $\ODown{s}_\leq$.
		We define a relation $\trianglelefteq$ with $u \trianglelefteq s$ if and only if $u \in X_s$. Let $\preceq$ be the reflexive and transitive closure of $\trianglelefteq$. Note that $\preceq$ respects $\leq$ by choice.
		
		Let $s \in S$ be an arbitrary element.
		First of all, we show that the set $\Down{s}_\preceq$ is finite. We construct a locally finite, rayless auxiliary tree $A$ covering all elements of $\Down{s}_\preceq$ recursively.
		Let $s$ be the root of $A$. In the $n$-th step we add the $n$-th distance class as follows: for every element $u$ in the $(n-1)$-st distance class we add all vertices of $X_u$ as a children of $u$. By construction, $A$ is locally finite. Furthermore, every branch of $A$ is $\leq$ descending. Since $\leq$ is a well-order, all branches of $A$ are finite. As $\preceq$ is the transitive closure of $\trianglelefteq$, all elements of $\Down{s}_\preceq$ are contained in $A$ and by K\H{o}nig's Infinity Lemma \cite{konig1927schlussweise}, $V(A) = \Down{s}_\preceq$ is finite. In particular, $\preceq$ is well-founded.
		
		Next, we show that the component $K$ in $\cK(S \setminus \ODown{s}_\preceq)$ containing $s$ is subset of $\Up{s}_\preceq$.
		Suppose for a contradiction that $K \setminus \Up{s}_{\preceq}$ is nonempty and let $k \in K \setminus \Up{s}_\preceq$ be $\preceq$-minimal. Note that $K \subseteq \Up{s}_{\leq}$ holds by the fact $X_s \subseteq \ODown{s}_\preceq$ and the choice of $X_s$. Thus $k > s$ holds. As $k$ and $s$ are contained in the connected set $K$, $X_k \cap K \neq \emptyset$. Note that $X_k \subseteq \ODown{k}_\preceq$ by the choice of $\preceq$.
		Thus $X_k \cap K \subseteq \ODown{k}_{\preceq} \cap K$ is nonempty. Furthermore, $\ODown{k}_\preceq \cap K \subseteq \Up{s}_\preceq$ holds by the choice of $k$.
		Thus $\ODown{k}_\preceq \cap \Up{s}_\preceq \neq \emptyset$ holds. This implies $k \in \Up{s}_\preceq$, contradicting the choice of $k$. Thus $K$ is indeed a subset of $\Up{s}_\preceq$.
		\item[\labelcref{itm:csn_equiv_3} implies \labelcref{itm:csn_equiv_1}]
		Let $\preceq$ be a well-founded partial order of $S$ fulfilling \labelcref{itm:csn_equiv_3i} and \labelcref{itm:csn_equiv_3ii}.
		By~\cref{prop:well-order_respecting}, there is a well-order $\leq$ of $S$ of order type $|S|$ respecting $\preceq$.
		Then $\leq$ witnesses countable separation number. \qedhere
	\end{description}
\end{proof}

\begin{prop} \label{prop:subconnectoid_closed}
	Let $(S, \cC)$ be a connectoid. If $(S, \cC)$ has countable separation number, then every subconnectoid of $(S, \cC)$ has countable separation number.
\end{prop}
\begin{proof}
	Note that for every subconnectoid $(S', \cC')$ of $(S, \cC)$ and for every set $X \subseteq S$ each component in $\cK'(S' \setminus (X \cap S'))$ is subset of a component in $\cK(S \setminus X)$, where $\cK'$ is the family of components of $(S', \cC')$. 
	Thus every well-order witnessing countable separation number for $(S, \cC)$ also witnesses countable separation number for every subconnectoid of $(S, \cC)$.
\end{proof}

Let $(S, \cC)$ be a connectoid and $\hat{S} \subseteq S$ be some subset. Furthermore, let $\lambda$ be some cardinal. We call $\hat S$ \emph{$\lambda$-strong} if for every $X \subseteq S \setminus \hat S$ with $|X| \leq \lambda$ and every finite connected set $C$ there is a finite connected set $C'$ such that $C' \cap X = \emptyset$ and $C \cap \hat S = C' \cap \hat S$. Furthermore, we call $\hat S$ $(< \alpha)$-strong for some cardinal $\alpha$ if $\hat S$ is $\lambda$-strong for every cardinal $\lambda < \alpha$.

Note that for cardinals $\alpha < \beta$ every $(< \beta)$-strong and every $\beta$-strong subset is $\alpha$-strong and in particular $(< \alpha)$-strong.

\begin{prop}\label{prop:torso}
	Let $(S, \cC)$ be a connectoid with countable separation number. Then every torso of $(S, \cC)$ at a $(< \omega)$-strong subset has countable separation number.
\end{prop}

\begin{proof}
	Let $\hat S$ be some $(< \omega)$-strong subset of $S$ and let $\leq$ be a well-order of $S$ witnessing countable separation number.
	We consider the torso $(\hat S, \hat \cC)$ and its family of components $\hat \cK$.
	Let $r \in \hat S$ be arbitrary.
	There is a finite set $X \subseteq S$ such that the component in $\cK(S \setminus X)$ containing $r$ does not contain an element of $\ODown{r}_\leq$.
	We show that the component $K \in \hat{\cK}(\hat S \setminus (X \cap \hat S))$ containing $r$ does not contain an element of $\ODown{r}_\leq \cap \hat S$.
	Then $\leq$ witnesses countable separation number for $(\hat S, \hat \cC)$.
	
	Suppose for a contradiction that $K \cap \hat S \cap \ODown{r}_\leq \neq \emptyset$.
	Let $C \in \cC$ be a connected set with $K = C \cap \hat S$.
	Note that $C \cap X \cap \hat{S} = \emptyset$.
	 Let $C' \subseteq C$ be a finite connected set in ${\cC}$ containing $r$ and some element of $\ODown{r}_\leq \cap \hat S$.
	Since $\hat S$ is $(< \omega)$-strong and as $X \setminus \hat S$ is finite, there is a connected set $C'' \in \cC$ with $C'' \cap (X \setminus \hat S) = \emptyset $ and $C' \cap \hat S = C'' \cap \hat S$.
	Thus $C''$ avoids $X$ and contains $r$ and some element of $\ODown{r}_{\leq}$, contradicting the choice of $X$.
\end{proof}

\begin{prop} \label{prop:contraction_closed}
	Let $(S, \cC)$ be a connectoid with countable separation number. Then every minor of $(S, \cC)$ has countable separation number.
\end{prop}
\begin{proof}
	By \cref{prop:subconnectoid_closed}, it remains to prove that countable separation number is maintained under contraction.
	Let $\mathcal{P}$ be some partition of $S$ into connected sets and let $(\mathcal{P}, \cC_{\mathcal{P}})$ be the minor of $(S, \cC)$ obtained by contracting $\mathcal{P}$. We denote the family of components of $(\mathcal{P}, \cC_{\mathcal{P}})$ as $\cK_\mathcal{P}$.
	
	Furthermore, let $ \leq$ be a well-order of $S$ witnessing countable separation number of $(S, \cC)$ and let $s_P$ be the $ \leq$-minimal element of $P$ for every $P \in \mathcal{P}$.
	We consider the well-order $ \leq '$ of $\mathcal{P}$ induced by $ \leq$ on $\{s_P:  P \in \mathcal{P}\}$ and prove that $\leq'$ witnesses countable separation number for $(\mathcal{P}, \cC_{\mathcal{P}})$.
	
	Let $P \in \mathcal{P}$ be arbitrary. Since $ \leq$ witnesses countable separation number for $(S, \cC)$, there is a finite set $X \subseteq \ODown{s_P}_\leq$ such that the component $K \in \cK(S \setminus X)$ containing $s_P$ avoids $\ODown{s_P}_\leq$. We consider the finite set $Y:= \{P \in \mathcal{P}: X \cap P \neq \emptyset \}$. By construction, $Y$ is a subset of $\{Q: Q <' P\}$.
	
	Suppose for a contradiction that the component $K' \in \cK_{\mathcal{P}}(\mathcal{P} \setminus Y)$ containing $P$ contains some $Q \in \mathcal{P}$ with $Q <' P$.
	Note that $\bigcup K' \subseteq K$ since $\bigcup Y \supseteq X$.
	Then $Q \subseteq K$ and in particular $s_Q \in K$.
	This gives a contradiction since $s_Q \in \ODown{s_P}_{\leq}$.
\end{proof}

Now we determine two classes of connected connectoids without countable separation number.
These two classes play a central role for the proof of \cref{thm:countable_separation_number}, as connectoids that are not of this type have certain adhesion and closure properties, which we investigate in \cref{sec:decomposition}.

We begin by showing a condition that ensures that a well-order does not witness countable separation number.

\begin{prop}\label{prop:forb}
Let $(S, \cC)$ be a connected connectoid and let $ \leq$ be a well-order of $S$. Further, let $(C_n)_{n \in \NN}$ be a family of connected sets with $\bigcap_{n \in \NN} C_n \neq \emptyset$. Let $z_n$ be the $ \leq$-minimal element of $C_n$ for every $n \in \NN$. If the elements of $(z_n)_{n \in \NN}$ are pairwise distinct, then $ \leq$ does not witness countable separation number. 
\end{prop}
\begin{proof}
Suppose for a contradiction that the elements $(z_n)_{n \in \NN}$ are pairwise distinct and that $\leq$ witnesses countable separation number.
Let $s$ be the $ \leq$-minimal element that is contained in infinitely many elements of $(C_n)_{n \in \NN}$.
There is a finite set $X \subseteq \ODown{s}_\leq$ such that the component $K \in \cK(S \setminus X)$ that contains $s$ avoids $\ODown{s}_\leq$. 
By the choice of $s$, the set $X$ hits only finitely many elements of $(C_n)_{n \in \NN}$.
Since the elements of $(z_n)_{n \in \NN}$ are pairwise distinct, there exists $n \in \NN$ such that $C_n$ contains $s$ and avoids the set $X$, and $z_n \neq s$.
This implies $z_n  \in \ODown{s}_\leq$ since $s \in C_n$ and further that $C_n \subseteq K$.
Thus $z_n \in C_n \cap \ODown{s}_\leq \subseteq K \cap \ODown{s}_\leq = \emptyset$, a contradiction.
\end{proof}

\subsection{Barricade}
Let $\lambda \geq \omega$ be a cardinal. A connected connectoid $(S, \cC)$ is a \emph{$(\lambda, \lambda^+)$-barricade} if there is $Y \subseteq S$ with $|Y| = \lambda$ and there is a family $(Z_i)_{i < \lambda^+}$ of nonempty, pairwise disjoint subsets of $S$ such that for every $i < \lambda^+$
\begin{itemize}
    \item $Z_i \cap Y = \emptyset$,
    \item for every finite set $X \subseteq S \setminus Z_i$, every component in $\cK(S \setminus X)$ containing an element of $Z_i$ contains an element of $Y$.
\end{itemize}

\begin{lemma} \label{lem:barricade}
Let $\lambda \geq \omega$ be some cardinal. Then no $(\lambda, \lambda^+)$-barricade $(S, \cC)$ has countable separation number.
\end{lemma}
\begin{proof}
Suppose for a contradiction that there is a well-order $ \leq$ of $S$ witnessing that $(S, \cC)$ has countable separation number. Let $i < \lambda^+$ be arbitrary. Let $z_i$ be the $ \leq$-minimal element of $Z_i$. Since $ \leq$ witnesses countable separation number, there is a finite set $X_i \subseteq \ODown{z_i}_\leq$ such that the component $K_i \in \cK(S \setminus X_i)$ containing $z_i$ avoids $ \ODown{z_i}_\leq$.
By choice of $z_i$, the set $X_i$ is disjoint to $Z_i$. Thus the component $K_i$ contains an element of $Y$ as $(S, \cC)$ is a $(\lambda, \lambda^+)$-barricade.
Take a finite connected set $C_i \subseteq K_i$ containing $z_i$ and an element of $Y$. 
Note that $z_i$ is the $ \leq$-minimal element of $C_i$ since $C_i \cap \ODown{z_i} \subseteq K_i \cap \ODown{z_i} = \emptyset$.

Since $|Y| = \lambda$, there is an element $y \in Y$ that is contained in infinitely many elements of $(C_i)_{i < \lambda^+}$. This infinite subfamily of  $(C_i)_{i < \lambda^+}$ contradicts \cref{prop:forb} since the elements of $(z_i)_{i < \lambda^+}$ are pairwise distinct by construction.
\end{proof}

In the following theorem we use a simple version of Fodor's Lemma:

\begin{lemma}[\cite{fodor1956bemerkung}] \label{lem:fodor}
	If $f: \omega_1 \rightarrow \omega_1$ is a function satisfying $f(\alpha) < \alpha$ for all $\alpha < \omega_1$, then $f$ is constant on an uncountable subset of $\omega_1$.
\end{lemma}

\begin{thm}\label{thm:countable_branches}
	Let $T$ be an order tree with an uncountable branch. Then every $T$-connectoid is an $(\omega, \omega_1)$-barricade.
\end{thm}
\begin{proof}
	Let $(S, \cC)$ be some $T$-connectoid. We consider an uncountable branch of $T$ and let $A$ be its initial segment of order type $\omega_1$. For $\alpha < \omega_1$ let $a_\alpha$ be the unique element of $A \cap T^\alpha$.
	Set $S_0:= \Up{a_0}_T \setminus \Up{a_{1}}_T$ and set $S_\alpha := ( \bigcap_{\beta < \alpha} \Up{a_{\beta+1}}_T) \setminus \Up{a_{\alpha+1}}_T$ for every $1 \leq \alpha < \omega_1$.
	Note that $S_\alpha = \Up{a_\alpha}_T \setminus \Up{a_{\alpha+1}}_T$ for every successors $\alpha < \omega_1$ and note further that $(S_\alpha)_{\alpha < \omega_1}$ is a partition of $S \setminus \bigcap_{\alpha < \omega_1} \Up{a_\alpha}_T$.
	We write $S_{[\alpha, \beta]}:= \bigcup_{\alpha \leq \gamma \leq \beta} S_\gamma$.
	
	We call a pair of ordinals $\alpha \leq \beta < \omega_1$ \emph{linkable} if there exists an uncountable set $\Gamma \subseteq \omega_1 \setminus \beta$ such that for every $\gamma \in \Gamma$ there is a connected set $C_\gamma$ containing $a_\alpha$ and $a_{\gamma}$ with $C_\gamma \subseteq S_{[\alpha, \beta]} \cup S_{\gamma}$.
	
	\begin{claim}\label{clm:countable_branch}
		For every $\alpha < \omega_1$ there is $\alpha \leq \beta < \omega_1$ such that $\alpha$ and $\beta$ are linkable.
	\end{claim}
	\begin{claimproof}
		By \cref{lem:interval}, for every $\gamma \geq \alpha$ there is a finite connected set $H_\gamma \subseteq S_{[\alpha, \gamma]}$ containing $a_\alpha$ and $a_{\gamma}$. Let $f(\gamma)$ be the maximal ordinal with $\gamma > f(\gamma) \geq \alpha$ such that $S_{f(\gamma)}$ contains an element of $H_\gamma$.
		This implies that $H_\gamma \subseteq S_{[\alpha, f(\gamma)]} \cup S_{\gamma}$.
		We apply Fodor's Lemma, \cref{lem:fodor}, to the function $f$. There is $\beta \geq \alpha$ and an uncountable set $O$ such that $f(\gamma) =\beta$ for every $\gamma \in O$.
		Then $\alpha$ and $\beta$ are linkable witnessed by the set $O$.
	\end{claimproof}
	
	Now we construct increasing sequences $(\alpha_i)_{i < \omega_1}$ and $(\beta_i)_{i < \omega_1}$ of ordinals such that
	\begin{itemize}
		\item the ordinals $\alpha_i$ and $\beta_i$ are linkable for every $i < \omega_1$, and
		\item the inequality $\alpha_i \leq \beta_i < \alpha_j \leq \beta_j$ holds for every $i < j < \omega_1$.
	\end{itemize}
	Let $i < \omega_1$ be some ordinal and suppose that $(\alpha_j)_{j < i}$ and $(\beta_j)_{j < i}$ have been defined.
	Then there exists $\alpha_i < \omega_1$ such that ${\alpha_j}, {\beta_j} < {\alpha_i}$ holds for every $j < i$.
	By \cref{clm:countable_branch}, there is $\beta_i < \omega_1$ such that $\alpha_i$ and $\beta_i$ are linkable.
	Then $(\alpha_j)_{j \leq i}$ and $(\beta_j)_{j \leq i}$ are as desired, which completes the construction.
	Note that the second property implies that the sets in $(S_{[\alpha_i, \beta_i]})_{i < \omega_1}$ are pairwise disjoint.
	
	We set $Y:= \{a_{\alpha_i} : i < \omega \}$. 
	For every $\omega \leq j < \omega_1$ there is an uncountable set $\Gamma_j$ of ordinals such that for every $\gamma \in \Gamma_j$ there is a connected set $C_j^\gamma$ containing $a_{\alpha_j}$ and $a_{\gamma}$ such that $C_j^\gamma \subseteq S_{[\alpha_j, \beta_j]} \cup S_{\gamma}$ since $\alpha_j$ and $\beta_j$ are linkable.
	Let $\mathfrak{C}_j:= \{C_j^\gamma: \gamma \in \Gamma_j\}$ and let $Z_j$ be the set of elements in $S$ that are contained in uncountably many elements of $\mathfrak{C}_j$. 
	
	 We prove that the set $Y$ and the sets in $(Z_j)_{\omega \leq j < \omega_1}$ witness that $(S, \cC)$ is an $(\omega, \omega_1)$-barricade.
	 Firstly, note that the connected sets of $\mathfrak{C}_j$ intersect only in $S_{[\alpha_j, \beta_j]}$ and thus $Z_j \subseteq S_{[\alpha_j, \beta_j]}$ holds for every $\omega \leq j < \omega_1$.
	Thus $Z_j \cap Z_k = \emptyset = Z_j \cap Y$ for every $j \neq k$ with $\omega \leq j, k < \omega_1$.
	
	 Secondly, let $\omega \leq j < \omega_1$ be some ordinal, let $z \in Z_j$ be arbitrary and let $X \subseteq S \setminus Z_j$ be some finite set.
	 We have to show that the component in $\cK(S \setminus X)$ containing $z$ contains an element of $Y$.
	 
	 Since $X$ is finite, there is $k < \omega_1$ such that $X \cap \bigcup_{k \leq j < \omega_1} S_{j} = \emptyset$ and there is $i < \omega$ such that $X \cap S_{[\alpha_i, \beta_i]} = \emptyset$.
	 We prove that $a_{\alpha_i}$ and $z$ are contained in the same component in $\cK(S \setminus X)$.
	 As $\alpha_i$ and $\beta_i$ are linkable, there is an ordinal $k \leq \ell < \omega_1$ and a connected set $H_1 \subseteq S_{[\alpha_i, \beta_i]} \cup S_{\ell}$ that contains $a_{\alpha_i}$ and $a_\ell$.
	 Thus $H_1$ avoids $X$.
	 
	 By construction of $Z_j$ and as $X \cap Z_j = \emptyset$, every element of $X$ is contained in at most countably many elements of $\mathfrak{C}_j$. Thus uncountably many elements of $\mathfrak{C}_j$ avoid $X$ and contain $z$. Pick $\gamma \in \Gamma_j$ with $ \gamma \geq \ell$ such that $C_j^\gamma$ avoids $X$ and contains $z$. Note that $C_j^\gamma$ contains $a_\gamma$.
	 
	 The elements $a_\ell, a_\gamma$ are contained in some connected set $H_2 \subseteq S_{[\ell, \gamma]}$ by \cref{lem:interval}. Since $S_{[\ell, \gamma]} \subseteq  \bigcup_{k \leq j < \omega_1} S_{j}$, $H_2$ avoids $X$.
	 Then the connected set $H_1 \cup H_2 \cup C_j^\gamma$ witness that $a_{\alpha_i}$ and $z$ are contained in the same component in $\cK(S \setminus X)$.
	 This finishes the proof.
\end{proof}

\begin{cor}\label{cor:t_connectoid_uncountable_branch}
	Let $(S, \cC)$ be $T$-connectoid for a normal order tree $T$.
	If there exists an uncountable branch in $T$, then $(S, \cC)$ does not have countable separation number.
\end{cor}
\begin{proof}
	If $T$ has an uncountable branch, then $(S, \cC)$ is a $(\omega, \omega_1)$-barricade by \cref{thm:countable_branches}.
	By \cref{lem:barricade}, $(S, \cC)$ does not have countable separation number.
\end{proof}

\begin{cor}\label{cor:normal_partition_tree_countable_branch}
	Let $(S, \cC)$ be a connected connectoid with countable separation number.
	Further, let $T$ be a normal partition tree of $(S, \cC)$. Then all branches of $T$ are countable.
\end{cor}
\begin{proof}
	Let $\mathcal{P}$ be the partition of the normal partition tree $T$.
	Since $(S, \cC)$ has countable separation number, also the minor $(\mathcal{P}, \cC_\mathcal{P})$ obtained by contracting $\mathcal{P}$ has countable separation number by \cref{prop:contraction_closed}.
	Then \cref{cor:t_connectoid_uncountable_branch} implies that all branches of $T$ are countable as $(\mathcal{P}, \cC_\mathcal{P})$ is a $T$-connectoid.
\end{proof}

\subsection{Aronszajn-tree}
In this subsection we investigate $T$-connectoids for order trees $T$ that contain a rooted Aronszajn-tree:
An Aronszajn-tree is an order tree $T$ for which $T^\alpha$ is nonempty and countable for every $\alpha < \omega_1$ and there are no uncountable branches in $T$.
We prove that such $T$-connectoids do not have countable separation number.

\begin{lemma}\label{lem:aron}
	Let $T$ be an order tree containing a rooted Aronszajn-tree and let $(S, \cC)$ be a $T$-connectoid. Then $(S, \cC)$ does not have countable separation number.
\end{lemma}
Diestel and Leader showed a similar result for graphs \cite{diestel2001normal}*{Proposition~3.6}:
An undirected graph $G$ is a \emph{$T$-graph} if its corresponding connectoid $(V(G), \cC)$ is a $T$-connectoid. They proved that for every Aronszajn-tree $T$ every $T$-graph does not have countable colouring number.
We adapt their proof to the more general notion of connectoids:
\begin{proof}
By \cref{cor:t_connectoid_uncountable_branch}, we can assume that all branches of $T$ are countable.
Let $\hat T$ be a rooted Aronszajn-tree in $T$.
Suppose for a contradiction that there is a well-order $ \leq$ of $S$ witnessing that $(S, \cC)$ has countable separation number.

First of all, we will choose for each limit ordinal $\alpha < \omega_1 $ some element $a_\alpha \in V(\hat T) \cap \bigcup_{\gamma < \alpha }T^\gamma$ and some element $b_\alpha \in \bigcup_{\gamma \geq \alpha} T^\gamma$ such that $a_\alpha$ and $b_\alpha$ are contained in a finite connected set $H_\alpha$ whose $\leq$-minimal element is $b_\alpha$.

Let $\alpha < \omega_1$ be an arbitrary limit ordinal and let $t_\alpha \in T^{\alpha} \cap V(\hat{T})$ be an arbitrary element.
By \cref{prop:construct_star} we obtain a family $(H_\alpha^n)_{n \in \NN}$ of finite connected sets containing $t_\alpha$ and some element of $\ODown{t_\alpha}_T$ such that $H_\alpha^n \cap H_\alpha^m \subseteq \bigcup_{\gamma \geq \alpha} T^\gamma$ for every $n \neq m \in \NN$.

Let $b_\alpha$ be the $ \leq$-minimal element that is contained in infinitely many sets of $(H_\alpha^n)_{n \in \NN}$.
Note that $b_\alpha \in \bigcup_{\gamma \geq \alpha} T^\gamma$ by construction of $(H_\alpha^n)_{n \in \NN}$.
Since $ \leq$ witnesses countable separation number, there is a finite set $X \subseteq \ODown{b_\alpha}_\leq$ such that the component in $\cK(S \setminus X)$ containing $b_\alpha$ avoids $\ODown{b_\alpha}_\leq$.
By choice of $b_\alpha$, every element $x \in X$ hits at most finitely many sets of $(H_\alpha^n)_{n \in \NN}$. Thus there is a connected set $H_\alpha$ in the family $(H_\alpha^n)_{n \in \NN}$ that avoids $X$ and contains $b_\alpha$. Then $b_\alpha$ is the $ \leq$-minimal element of $H_\alpha$ by the choice of $X$.
Let $a_\alpha$ be some element of $H_\alpha \cap \ODown{t_\alpha}_T$, which exists by construction.
Note that $a_\alpha \in \hat{T}$ since $t_\alpha \in \hat{T}$.
Thus $a_\alpha, b_\alpha$ and $H_\alpha$ are as desired.

For every limit ordinal $\alpha$ we define $f(\alpha)$ to be the limit ordinal such that $a_\alpha \in T^{f(\alpha) + n}$ for some $n \in \NN$. Note that $f(\alpha) < \alpha$ holds by construction.
We apply Fodor's Lemma, \cref{lem:fodor}, to the function $f$ to obtain a limit ordinal $\gamma$ and an uncountable set $O$ of limit ordinals such that $f(\alpha) = \gamma$ for every $\alpha \in O$.
Since $a_\alpha \in \hat{T}$ for every $\alpha \in O$ and as the levels of $\hat{T}$ are countable, there is an uncountable subset $O' \subseteq O$ and $a \in \hat{T}$ such that $a_\alpha = a$ for every $\alpha \in O'$.

We construct recursively an infinite subset $O'' \subseteq O'$ such that the elements of $(b_\alpha)_{\alpha \in O''}$ are pairwise distinct.
Assume that a finite set $\tilde{O}\subseteq O'$ of ordinals with the desired property has been defined.
Since $\tilde{O}$ is finite, there exists a limit ordinal $o \in O'$ such that $\{b_{\tilde{o}}: \tilde{o} \in \tilde{O} \} \cap \bigcup_{\gamma \geq o} T^\gamma = \emptyset$.
Then $b_o \in \bigcup_{\gamma \geq o} T^\gamma $ is not in $\{b_{\tilde{o}}: \tilde{o} \in \tilde{O} \}$ and this implies that $\tilde{O} \cup \{o\}$ is as desired, which completes the construction of $O''$.
Thus $(H_\alpha)_{\alpha \in O''}$, $(b_\alpha)_{\alpha \in O''}$ and the element $a$ contradict \cref{prop:forb}.
\end{proof}

A subtree $T'$ of an order tree $T$ is \emph{upright} if the tree order of $T'$ coincides with the tree order of $T$.

\begin{cor}\label{cor:aron}
	Let $(S, \cC)$ be a connected connectoid and let $T$ be a normal partition tree of $(S, \cC)$. If $T$ contains an upright Aronszajn-tree, then $(S, \cC)$ does not have countable separation number.
\end{cor}
\begin{proof}
	Let $\mathcal{P} = \{P_t: t \in T \}$ be the partition of $S$ such that $(\mathcal{P}, \cC_{\mathcal{P}})$ is a $T$-connectoid.
	It suffices to show that the minor $(\mathcal{P}, \cC_{\mathcal{P}})$ does not have countable separation number by \cref{prop:contraction_closed}.
	
	By \cref{cor:t_connectoid_uncountable_branch}, we can assume that all branches of $T$ are countable.
	Let $\hat T$ be an upright Aronszajn-tree in $T$ and let $r$ be the root of $\hat T$.
	Furthermore, let $\Tilde T$ be the order tree obtained by restricting $T$ to $\hat T$ and $\ODown{r}_T$.
	Since all branches of $T$ are countable, all branches of $\Tilde T$ are countable.
	Moreover, all level of $\Tilde T$ are countable, since $\hat T$ is an Aronszajn-tree.
	This implies that $\Tilde T$ is a rooted Aronszajn-tree in $T$, and thus $(\mathcal{P}, \cC_{\mathcal{P}'})$ does not have countable separation number by~\cref{lem:aron}.
\end{proof}

\section{Decomposition into subconnectoids of finite adhesion} \label{sec:decomposition}
In this section we prove that every connected connectoid $(S, \cC)$ with countable separation number can be decomposed in a specific way (see below).
This decomposition property is the central tool in the proof of \cref{thm:countable_separation_number}, as it enables us to construct the desired normal spanning tree recursively.

We recall some definitions.
Let $(S, \cC)$ be a connectoid.
Given a partition tree $T$ with a partition $\{P_t: t \in T\}$ of $S$ into elements of $\cC$, we set $P(U):= \bigcup_{u \in U} P_u$ for every $U \subseteq V(T)$ and $P(T'):= P(V(T'))$ for every subtree $T'$ of $T$.

Given a set $\hat S \subseteq S$, a component $K \in \cK(S \setminus \hat S)$ has \emph{finite adhesion} to $\hat S$ if there exists a finite set $X \subseteq \hat S$ such that $K$ is a component in $\cK(S \setminus X)$.
Otherwise, the component $K$ has \emph{infinite adhesion} to $\hat S$.
Furthermore, we call $\hat S$ \emph{$\lambda$-strong} for some cardinal $\lambda$ if for every $X \subseteq S \setminus \hat S$ with $|X| \leq \lambda$ and every finite connected set $C$ there is a finite connected set $C'$ such that $C' \cap X = \emptyset$ and $C \cap \hat S = C' \cap \hat S$.
Furthermore, we call $\hat S$ $(< \alpha)$-strong for some cardinal $\alpha$ if $\hat S$ is $\lambda$-strong for every cardinal $\lambda < \alpha$.

\begin{restatable}{lemma}{DecompositionLemma} \label{thm:decomp}
	Let $(S, \cC)$ be a connected connectoid that has countable separation number and $\kappa:= |S| \geq \omega_1$. Then every normal partition tree $T$ of $(S, \cC)$ can be written as $T = \bigcup_{i < \cf(\kappa)} T_i$ for a continuous increasing sequence of rooted subtrees $(T_i)_{i < \cf(\kappa)}$ such that for every $i < \cf(\kappa)$
	\begin{itemize}
		\item $\omega \leq |P(T_i)| < \kappa$,
		\item every component in $\cK(S \setminus P(T_i))$ has finite adhesion to $P(T_i)$, and
		\item $ P(T_i)$ is $|P(T_i)|$-strong.
	\end{itemize}
\end{restatable}

Pitz proved a similar result for undirected graphs using the weaker assumption that every minor has countable colouring number \cite{pitz2021proof}*{Theorem~3.6}.
We adapt his proof technique to the setting of connectoids.

We begin by proving an adhesion property of connected connectoids with countable separation number in \cref{sec:adhesion_property}.
We deduce that subsets of normal partition trees can be closed in a specific way in \cref{sec:closure_property}.
Finally, we construct the desired decomposition for \cref{thm:decomp} by recursion in \cref{sec:decomp}.

\subsection{Adhesion property of connected connectoids that are not barricades}\label{sec:adhesion_property}

\begin{lemma}\label{lem:adhesion}
Let $\lambda \geq \omega$ be a cardinal, let $(S, \cC)$ be a connected connectoid that is not a $(\lambda, \lambda^+)$-barricade and let $Y \subseteq S$ with $|Y| = \lambda$. Then there exists $Z \subseteq S \setminus Y$ with $|Z| \leq \lambda$ such that every component in $\cK(S \setminus Y)$ with finite adhesion to $Y$ avoids $Z$, and every component in $\cK(S \setminus Y)$ avoiding $Z$ has finite adhesion to $Y \cup Z$.
\end{lemma}
Before we prove \cref{lem:adhesion} we remark that for every $X \subseteq Y \subseteq S$ every component in $\cK(S \setminus X) \cap \cK(S \setminus Y)$ with finite adhesion to $X$ also has finite adhesion to $Y$. 
\begin{proof}
We construct a family $(Z_i)_{i < \sigma}$ of nonempty, pairwise disjoint subsets of $S$ such that
\begin{enumerate}[label=(\roman*)]
    \item\label{itm:adhesion_1} $|Z_i| \leq \omega$,
    \item\label{itm:adhesion_2} $Z_i \cap Y = \emptyset$, and
    \item\label{itm:adhesion_3} for every finite set $X \subseteq S \setminus Z_i$, every component in $\cK(S \setminus X)$ containing an element of $Z_i$ contains an element of $Y$
\end{enumerate}
for every $i < \sigma$ by adding new sets $Z_i$ recursively for as long as possible.

If $\sigma \geq \lambda^+$, then $(Z_i)_{i < \lambda^+}$ and $Y$ witness that $(S, \cC)$ is a $(\lambda, \lambda^+)$-barricade, contradicting the assumption. Therefore we can assume $\sigma < \lambda^+$.
Thus the construction terminates after $\lambda$ many steps and the set $Z:= \bigcup_{i < \sigma} Z_i$ has size at most $\lambda$.
Further, every component $K \in \cK(S \setminus Y)$ with finite adhesion to $Y$ is an element of $\cK(S \setminus X)$ for some finite $X \subseteq Y \subseteq S \setminus Z$.
As $K$ avoids $Y$, the component $K$ does not contain an element of $Z$ by \labelcref{itm:adhesion_3}.
Thus every component in $\cK(S \setminus Y)$ with finite adhesion to $Y$ avoids $Z$.

It remains to prove that every component in $\cK(S \setminus Y)$ avoiding $Z$ has finite adhesion to $Y \cup Z$.
Suppose for a contradiction that there exists a component $K \in \cK(S \setminus Y)$ avoiding $Z$ that has infinite adhesion to $Y \cup Z$.
We show that we can extend the family $(Z_i)_{i < \sigma}$ by an element $Z_\sigma$, contradicting the maximality of this family.

Let $b \in K$ be arbitrary. As $K$ has infinite adhesion to $Y \cup Z$, there is a family $(H_n)_{n \in \NN}$ of finite connected sets with $b \in H_n$, $H_n \cap (Y \cup Z) \neq \emptyset$ and $H_n \cap H_m \cap (Y \cup Z) = \emptyset$ for every $n \neq m \in \NN$ by~\cref{prop:infinite_adhesion_witness}.
Since $K$ is a component in $\cK(S \setminus Y)$, $H_n \cap (S \setminus K) \supseteq H_n \cap (Y \cup Z) \neq \emptyset$ implies $H_n \cap Y \neq \emptyset$ for every $n \in \NN$.

Let $Z_\sigma$ be the set of elements that are contained in infinitely many elements of $(H_n)_{n \in \NN}$.
Note that $b \in Z_\sigma$.
By construction,  $Z_\sigma \subseteq S \setminus (Y \cup Z)$ holds.
Thus $(Z_i)_{i \leq \sigma}$ is a family of nonempty, pairwise disjoint subsets of $S$ and \labelcref{itm:adhesion_2} is satisfied for $\sigma$.
Furthermore, \labelcref{itm:adhesion_1} is satisfied for $\sigma$ since $|\bigcup_{n \in \NN} H_n | = \omega$ and $Z_\sigma \subseteq \bigcup_{n \in \NN} H_n$.

Let $z \in Z_\sigma$ be arbitrary and let $X \subseteq S \setminus Z_\sigma$ be some finite set.
By choice of $Z_\sigma$, every $x \in X \subseteq S \setminus Z_\sigma$ is contained in at most finitely many connected sets of $(H_n)_{n \in \NN}$. Thus the set $X$ intersects only finitely many elements of $(H_n)_{n \in \NN}$.
Since infinitely many elements of $(H_n)_{n \in \NN}$ contain $z$, there is $n \in \NN$ such that $z \in H_n$ and $H_n$ avoids $X$.
We already observed that $H_n$ contains an element of $Y$.
Thus the component in $\cK(S \setminus X)$ containing $z$ also contains an element of $Y$.
Since $z$ and $X$ were chosen arbitrarily, \labelcref{itm:adhesion_3} holds for $\sigma$.
This completes the proof.
\end{proof}

\subsection{Closure lemma}\label{sec:closure_property}
\begin{lemma} \label{thm:closure}
Let $S$ be an uncountable set and let $(S, \cC)$ be a connected connectoid with countable separation number.
Furthermore, let $T$ be a normal partition tree of $(S, \cC)$.
Then for every infinite set $X \subseteq V(T)$ there is a rooted subtree $T' \subseteq T$ with $X \subseteq V(T')$ and $|V(T')| = |X|$ such that every component in $\cK(S \setminus P(T'))$ has finite adhesion to $P(T')$.
\end{lemma}
\begin{proof}
We recursively build an increasing sequence $(T_i)_{i \leq \omega_1}$ of rooted subtrees of $T$. Let $T_0:= T[\Down{X}_T]$, i.e.\ the minimal rooted subtree of $T$ that contains $X$. We assume that some $T_i$ for $i < \omega_1$ has been defined.
By \cref{lem:barricade}, $(S, \cC)$ is not a $(|P(T_i)|, |P(T_i)|^+)$-barricade since $(S, \cC)$ has countable separation number.
Thus, by \cref{lem:adhesion}, there exists $Z_i \subseteq S \setminus P(T_i)$ such that $|Z_i| \leq |P(T_i)|$, every component in $\cK(S \setminus P(T_i))$ with finite adhesion to $P(T_i)$ avoids $Z_i$, and every component in $\cK(S \setminus  P(T_i))$ avoiding $Z_i$ has finite adhesion to $P(T_i) \cup Z_i$. Let $T(Z_i):= \Set{t \in T: P_t \cap Z_i \neq \emptyset}$ and set
\[
T_{i+1}:= T_i \cup \Down{T(Z_i)}_T.
\]
For limits $\alpha \leq \omega_1$ we set $T_\alpha:= \bigcup_{i < \alpha} T_i$  and let $T' := T_{\omega_1}$.

\begin{claim}\label{clm:closure_1}
Every component in $\cK(S \setminus P(T_i))$ with finite adhesion to $P(T_i)$ is also a component in $\cK(S \setminus P(T'))$ with finite adhesion to $P(T')$.
\end{claim}
\begin{claimproof}
	Let $i < \omega_1$ be arbitrary and let $K \in \cK(S \setminus P(T_i))$ be some component with finite adhesion to $P(T_i)$.
We show by transfinite induction that $K$ is a component in $\cK(S \setminus P(T_j))$ with finite adhesion to $P(T_j)$ for every $i < j \leq \omega_1$.

We assume that $K$ is a component in $\cK(S \setminus P(T_j))$ with finite adhesion to $P(T_j)$ for some $i \leq j < \omega_1$.
Note that $K = P(\Up{s}_T)$ for some $s \in T - T_j$ by \cref{lem:tdigraph_property}.
By choice of $Z_j$, $K$ avoids $Z_j$. 
Then $T(Z_j) \cap \Up{s}_T = \emptyset$ holds by the definition of $T(Z_j)$ and thus $\Down{T(Z_j)}_T \cap \Up{s}_T = \emptyset$.
This implies that $\Up{s}_T$ avoids $T_{j + 1}$ and therefore $K = P(\Up{s}_T)$ is a component in $\cK(S \setminus P(T_{j + 1}))$. 
In particular, $K$ has finite adhesion to $P(T_j)$ since $P(T_j)$ is a superset of $P(T_i)$.

We assume that $K$ is a component in $\cK(S \setminus P(T_{j'}))$ for every $i \leq j' <j$, where $j \leq \omega_1$ is an arbitrary limit.
Then $K$ avoids $P(T_{j'})$ for every $i \leq j' < j$.
Thus $K$ avoids $\bigcup_{i \leq j' < j} P(T_{j'}) = P(T_j)$.
Then $K$ is a component in $\cK(S \setminus P(T_j))$ since $K$ is a component in $\cK(S \setminus P(T_i))$.
The component $K$ has finite adhesion to $P(T_j)$ since $P(T_j)$ is a superset of $P(T_i)$.

Thus $K$ is indeed a component in $\cK(S \setminus P(T'))$ with finite adhesion to $P(T')$.
\end{claimproof}

Note that $\ODown{t}_T \cap V(T_0)$ forms an initial segment of $\ODown{t}_T$ for every $t \in T' - T_0$ since $T_0$ is a rooted subtree of $T$.
We set $h(t)$ to be the order type of $\ODown{t}_T \setminus V(T_0)$. Note that $h(t) \leq \height(t) < \omega_1$, by \cref{cor:normal_partition_tree_countable_branch}.
\begin{claim}\label{clm:closure_2}
For every $t \in T' - T_0$ we have $t \in T_{h(t)+1}$.
\end{claim}
\begin{claimproof}
Suppose not and pick $t \in T'$ minimal regarding $h(t)$ with the property $t \notin T_{h(t) + 1}$.
Then $\ODown{t}_T \subseteq V(T_{h(t)})$ holds.
We consider the minor $(V(T), \cC_T)$ obtained by contracting the partition $\{P_t: t \in T\}$ and we denote its components by $\cK_T$.
Note that $t \notin V(T_{h(t)}), V(T_{h(t)+1})$ but $\ODown{t}_T \subseteq V(T_{h(t)}), V(T_{h(t)+ 1})$.
Thus the set $\Up{t}_T$ is a component in $\cK_T( T - T_{h(t)}) \cap \cK_T(T - T_{h(t)+ 1})$ by \cref{lem:tdigraph_property}.
Thus $\Up{t}_T \cap T(Z_{h(t)}) = \emptyset$ holds by construction.

We rephrase these properties in the setting of $(S, \cC)$: The set $P(\Up{t}_T)$ is a component in $\cK(S \setminus P(T_{h(t)})) \cap \cK(S \setminus P(T_{h(t) +1}))$ and $P(\Up{t}_T) \cap Z_{h(t)}  = \emptyset$ holds.
By choice of $Z_{h(t)}$, the component $P(\Up{t}_T) \in \cK(S \setminus P(T_{h(t)}))$ has finite adhesion to $P(T_{h(t)}) \cup Z_{h(t)}$.
Thus $P(\Up{t}_T)$ has finite adhesion to $P(T_{h(t)+1})$ since $P(T_{h(t)+1})$ is a superset of $P(T_{h(t)}) \cup Z_{h(t)}$.
Then $P(\Up{t}_T)$ is a component in $\cK(S \setminus P(T'))$ by \cref{clm:closure_1}, contradicting $t \in T'$.
\end{claimproof}

By construction, $T'$ is a rooted subtree of $T$ with $X \subseteq V(T')$. It remains to prove that every component in $\cK(S \setminus P(T'))$ has finite adhesion to $P(T')$ and that $|V(T')|=|X|$ holds.

Let $K$ be an arbitrary component in $\cK(S \setminus P(T'))$. Then $K = P(\Up{t}_T)$ for some $t \in T - T'$ such that $\ODown{t}_T \subseteq V(T')$ by \cref{lem:tdigraph_property}.
Then $\ODown{t}_T \subseteq V(T_{h(t)})$ holds by \cref{clm:closure_2} and the fact that $h(t) > h(s)$ for every $s \in \ODown{t}_T \setminus V(T_0)$. 
Note that $K$ is a component in $\cK(S \setminus P(\ODown{t}_T))$ since $T$ is a normal partition tree.
Thus $K$ is a component in $\cK(S \setminus P(T_{h(t)})) \cap \cK(S \setminus P(T_{h(t)+1}))$.
This implies $T(Z_{h(t)}) \cap \Up{t}_T = \emptyset$, which is equivalent to $Z_{h(t)} \cap K = \emptyset$.
By choice of $Z_{h(t)}$, $K$ has finite adhesion to $P(T_{h(t)}) \cup Z_{h(t)}$.
Since $P(T')$ is a superset of $P(T_{h(t)}) \cup Z_{h(t)}$, the component $K$ has finite adhesion to $P(T')$.

Now we show by transfinite induction that $|T_i| = |X|$ for every $i < \omega_1$. For limits, this property is clearly maintained.
We assume that $|T_i| = |X|$ for some $i < \omega_1$. Note that all branches of $T_i$ and in particular all sets $P_t$ for $t \in T_i$ are countable by \cref{cor:normal_partition_tree_countable_branch}.
Thus $P(T_{i})$ has size $|T_i|=|X|$.
By choice of $Z_i$, the set $Z_i$ has size at most $|X|$ and therefore $T(Z_i)$ has size at most $|X|$. As the branches of $T$ are countable, also $\Down{T(Z_i)}_T$ has size at most $|X|$.
We deduce that $T_{i + 1}$ has size $|X|$.

Finally we prove that $T'$ has size $|X|$. If $|X|$ is uncountable, this holds true, since $|T'| =|\bigcup_{i < \omega_1} T_i| \leq \omega_1 \cdot |X|$. Suppose for a contradiction, that $|X| = \omega$ and $|T'| = \omega_1$ hold. We consider the set $Y:= \Set{t \in T' - T_0: h(t)=0}$.
The set $Y$ is a subset of $T_1$, by \cref{clm:closure_2}, and thus the set $Y$ contains at most countably many elements.
Since $T' - T_0 \subseteq \bigcup_{y \in Y} \Up{y}_T \cap T'$, there is $y \in Y$ such that $|\Up{y}_T \cap T'| = \omega_1$.
Set $T'':= T'[\Up{y}_T]$.
Note that $T''$ is an upright subtree of $T$ rooted at $y$.
The branches of $T''$ are countable, as the branches of $T$ are countable.
For every $t \in T''$ the height of $t$ with respect to the subtree $T''$ coincides with $h(t)$.
Thus the $i$-th level of $T''$ is contained in $T_{i + 1}$, by \cref{clm:closure_2}, for every $i < \omega_1$ and therefore countable.
Since $|T''|= \omega_1$, the $i$-th level of $T''$ has to be nonempty for every $i < \omega_1$.
Thus $T''$ is an Aronszajn-tree, which contradicts \cref{cor:aron}.
\end{proof}

\subsection{Decomposition Lemma}\label{sec:decomp}

\DecompositionLemma*

\begin{proof}[Proof of \cref{thm:decomp} for $\kappa$ regular]
We consider the partition $\mathcal{P}:=\{P_t: t \in T\}$ and the minor $(\mathcal{P}, \cC_\mathcal{P})$ obtained by contracting $\mathcal{P}$.
By \cref{cor:normal_partition_tree_countable_branch}, we have $|\mathcal{P}|  = |T| = |S| = \kappa$.
By \cref{prop:contraction_closed} and \cref{lem:countable_separtion_number_equivalence}, there exists a well-order $\leq$ of $\mathcal{P}$ of order type $\kappa$ witnessing that $(\mathcal{P}, \cC_\mathcal{P})$ has countable separation number.
We pick an enumeration $V(T) = \Set{t_i : i < \kappa}$.
 
 Next, we recursively construct a continuous increasing sequence $(T_i)_{i < \kappa}$ of rooted subtrees of $T$ such that for each $i < \kappa$
\begin{enumerate}[label=(\alph*)]
	\item\label{itm:decomp_4} $t_i \in T_{i + 1}$,
    \item\label{itm:decomp_1} $\omega \leq |P(T_i)| < \kappa$,
    \item\label{itm:decomp_5} the set $V(T_i)$ is an initial segment of $(\mathcal{P}, \leq )$,
    \item\label{itm:decomp_2} every component in $\cK(S \setminus P(T_i))$ has finite adhesion to $P(T_i)$, and
    \item\label{itm:decomp_3} $P(T_i)$ is $(< \kappa)$-strong.
\end{enumerate}
Then the sequence $(T_i)_{i < \kappa}$ is as desired.

Let $T_{-1}$ be some infinite, countable rooted subtree of $T$ and let $T_\alpha := \bigcup_{i < \alpha} T_i$ for every limit $\alpha < \kappa$ such that every $T_i$ with $i < \alpha$ has been defined.
We assume that $T_i$ has been defined for some $i \in \{-1\} \cup \kappa$. We set $T_i^0 := T_i \cup \Down{t_i}_T$.
Now we construct recursively an increasing sequence $(T_i^n)_{n \in \NN_0}$ of subtrees of $T$ using the following three steps.
\begin{itemize}
	\item If  $T_i^{3n}$ has been defined for some $n \in \NN_0$, let $T_i^{3n + 1} \subseteq T$ be a rooted tree of size $|T_i^{3n}|$ containing $T_i^{3n}$ such that every component in $\cK(S \setminus P(T_i^{3n + 1}))$ has finite adhesion to $P(T_i^{3n + 1})$, which exists by \cref{thm:closure}.
	\item If $T_i^{3n+1}$ has been defined for some $n \in \NN_0$, let $T_i^{3n + 2}$ be the smallest rooted subtree of $T$ containing the $\leq$-down-closure of $T_i^{3n + 1}$.
	\item If $T_i^{3n + 2}$ has been defined for some $n \in \NN_0$, we construct a set $Y \subseteq S$.
	We begin by setting $Y:= \emptyset$. For every finite subset $Z \subseteq P(T_i^{3n + 2})$ we consider a maximal family $\mathcal{F}(Z)$ of finite connected sets that intersect $P(T_i^{3n +2})$ exactly in $Z$ such that every pair of distinct elements in $\mathcal{F}(Z)$ is disjoint outside $Z$. If $|\mathcal{F}(Z)| < \kappa$ holds, we add $\bigcup \mathcal{F}(Z)$ to $Y$. Note that $|Y| < \kappa$ holds since $\kappa$ is regular.
	
	\noindent
	Let $T(Y):= \{t \in T: P_t \cap Y \neq \emptyset \}$ and set $T_i^{3n + 3}:= T_i^{3n + 2} \cup \Down{T(Y_i^{3n + 2})}_T$.
\end{itemize}
Finally, we set $T_{i + 1}:= \bigcup_{m \in \NN} T_i^m$.

By construction, $t_i$ is an element of $T_{i + 1}$ for every $i < \kappa$, which proves \labelcref{itm:decomp_4}.
We show \labelcref{itm:decomp_1} by transfinite induction.
If some $i < \kappa$ satisfies  \labelcref{itm:decomp_1}, we show that all three steps in the construction of $(T_i^n)_{n \in \NN_0}$ ensure that the trees in $(T_i^n)_{n \in \NN_0}$ have size $< \kappa$.
For the first step, $|T_i^{3n}|= |T_i^{3n+1}|$ holds for every $n \in \NN$ by construction.
For the second step, note that $\leq$ is of order type $\kappa$ and therefore every $\leq$-down-closure of a set of a size $< \kappa$ has size $< \kappa$.
For the third step, note that the set $\Down{T(Y_i^{3n + 2})}_T$ has size $\omega \cdot |Y|< \kappa$ since all branches of $T$ are countable by \cref{cor:normal_partition_tree_countable_branch}.

Since $\kappa$ is regular, the union $T_{i + 1}= \bigcup_{m \in \NN} T_i^m$ also has size $<\kappa$.
By the same argument, the tree $T_i$ have size $< \kappa$ for some limit $i < \kappa$ if all $T_j$ with $j < i$ has size $< \kappa$.
Then \labelcref{itm:decomp_1} is satisfied for every $i < \kappa$ as all sets $P_t $ with $t \in T$ are countable by \cref{cor:normal_partition_tree_countable_branch}.

By construction, $P(T_{i+1})$ is an initial segment of $(\mathcal{P}, \leq)$ for every $i < \kappa$.
Since $(T_i)_{i < \kappa}$ is a continuous sequence, $P(T_i)$ also forms an initial segment of $(\mathcal{P}, \leq)$ for every limit $i < \kappa$, which proves \labelcref{itm:decomp_5}.

We turn our attention to the proof of \labelcref{itm:decomp_2}.
Let $K$ be an arbitrary component in $\cK(S \setminus P(T_{i}))$ for some $i < \kappa$.
Then $K = P(\Up{t}_T)$ for some $t \in T$ with $\ODown{t}_T \subseteq V(T_{i})$ by \cref{lem:tdigraph_property}.

If there exist $s \in T_i$ with $\Down{s}_T \supseteq \ODown{t}_T$, then there is $j < i$ and $n \in \NN$ such that $s \in T_j^{3n +1}$ and therefore $\ODown{t}_{T} \subseteq T_j^{3n +1}$.
Then $K$ is a component in $\cK(S \setminus P(T_{j}^{3n + 1}))$ since $T$ is a normal partition tree.
This implies that $K$ has finite adhesion to $P(T_j^{3n + 1})$ by construction of $T_j^{3n + 1}$.
Thus $K$ has finite adhesion to $P(T_{i})$ since $P(T_{i})$ is a superset of $P(T_j^{3n + 1})$.

Otherwise every $s \in T_i$ has the property that $\Down{s}_T \not\supseteq \ODown{t}_T$, which implies $V(T_i) \cap \bigcap_{u \in \ODown{t}_T} \Up{u}_T = \emptyset$.
In particular, $t$ is a limit in $T$.
We apply \cref{prop:construct_star} to obtain a family $(H_n)_{n \in \NN}$ of finite connected sets in $\cC_{\mathcal{P}}$ with $t \in H_n$, $H_n \cap \ODown{t}_T \neq \emptyset$ and $H_n \cap H_m \subseteq \bigcap_{u \in \ODown{t}_T} \Up{u}_T$ for every $n \neq m \in \NN$.
Note that the $ \leq$-minimal element $b_n$ of $H_n$ is contained in $V(T_{i})$, since $V(T_{i})$ is a $ \leq$-initial segment and as $\emptyset \neq H_n \cap \ODown{t}_T \subseteq H_n \cap V(T_{i})$.
Thus $b_n \neq b_m$ for every $n \neq m \in \NN$ since $H_n \cap H_m \subseteq \bigcap_{u \in \ODown{t}_T} \Up{u}_T \subseteq \mathcal{P} \setminus V(T_i)$.
Then applying~\cref{prop:forb} to the connected sets $(H_n)_{n \in \NN}$ shows that $\leq$ does not witness that $(\mathcal{P},\cC_{\mathcal{P}})$ has countable separation number, a contradiction.
Thus \labelcref{itm:decomp_2} holds for every $i < \kappa$.

Finally, we prove that \labelcref{itm:decomp_3} is satisfied for every $i < \kappa$.
Let $C$ be an arbitrary finite connected set and let $X \subseteq S \setminus P(T_i)$ be some set with $|X| < \kappa$.
We have to show that there exists a finite connected set $C' \subseteq S \setminus X$ with $C \cap P(T_i) = C' \cap P(T_i)$.
If $C \subseteq P(T_i)$, then the connected set $C':= C$ is as desired.

Thus we can assume that $C \setminus P(T_i) \neq \emptyset$.
Set $Z:= C \cap P(T_i)$, and let $j < i$ and $n \in \NN$ such that $Z \subseteq P(T_j^{3n + 2})$.
For the construction of $T_j^{3n + 3}$ we considered a maximal family $\mathcal{F}(Z)$ of finite connected sets intersecting $P(T_i^{3n +2})$ exactly in $Z$ such that every pair of distinct elements of $\mathcal{F}(Z)$ is disjoint outside $Z$.
By maximality of $\mathcal{F}(Z)$,
$(C \setminus Z) \cap \bigcup \mathcal{F}(Z)  \neq \emptyset$ holds.
We did not add $\bigcup \mathcal{F}(Z)$ to $Y$ in the construction of $T_j^{3n + 3}$, as $C \setminus Z \subseteq S \setminus P(T_i) \subseteq S \setminus P(T_j^{3n + 3})$ holds.
Thus the family $\mathcal{F}(Z)$ has size at least $\kappa$.

As $|P(T_i)| < \kappa$ and $|X| < \kappa$, there is a connected set $C' \in \mathcal{F}(Z)$ such that $(C' \setminus Z) \cap P(T_i) = \emptyset$ and $C' \cap X = \emptyset$.
Thus $C'$ is as desired, which shows that \labelcref{itm:decomp_3} is satisfied.
\end{proof}

\begin{prop}\label{prop:finite_adhesion_uncountable}
	Let $(S, \cC)$ be a connected connectoid and $(S_i)_{i < \omega_1}$ an increasing sequence of subsets of $S$.
	Let $K$ be a component in $\cK(S \setminus \bigcup_{i < \omega_1}S_i)$ such that the component in $\cK(S \setminus S_i)$ containing $K$ has finite adhesion to $S_i$ for every $i < \omega_1$.
	Then $K$ has finite adhesion to $\bigcup_{i < \omega_1}S_i$.
\end{prop}
\begin{proof}
	Suppose for a contradiction that $K$ has infinite adhesion to $\bigcup_{i < \omega_1}S_i$.
	Then there is a family $(H_n)_{n \in \NN}$ of finite connected sets such that $H_n \cap H_m \cap \bigcup_{i < \omega_1}S_i = \emptyset$, $H_n \cap \bigcup_{i < \omega_1}S_i \neq \emptyset$ and $H_n \cap K \neq \emptyset$ for every $n \neq m \in \NN$ by~\cref{prop:infinite_adhesion_witness}.
	Since $\bigcup_{n \in \NN} H_n$ is countable, there is $j < \omega_1$ such that $\bigcup_{i < \omega_1}S_i \cap \bigcup_{n \in \NN} H_n \subseteq S_j$.
	This contradicts the fact that the component in $\cK(S \setminus S_j)$ containing $K$ has finite adhesion to $S_j$.
\end{proof}

\begin{proof}[Proof of \cref{thm:decomp} for $\kappa$ singular]
We fix an enumeration $V(T) = \Set{t_i : i < \kappa}$ and a continuous increasing sequence $\Set{\kappa_i: i < \cf(\kappa)}$ of uncountable cardinals with limit $\kappa$ such that $\kappa_0 > \cf(\kappa)$.
First of all, we construct a family $(T_{i,j})_{i < \cf(\kappa), j < \omega_1}$ of rooted subtrees of $T$ with $|T_{i,j}| = \kappa_i$ such that every component in $\cK(S \setminus P(T_{i,j}))$ has finite adhesion to $P(T_{i,j})$.

For every $i < \cf(\kappa)$ and $j < \omega_1$ the tree $T_{i,j}$ gets equipped with an enumeration of its vertex set $\Set{t_{i, j}^k: k < \kappa_i}$ and further, we define a set $Y_{i,j} \subseteq S$ in the following way:
We begin by setting $Y_{i,j}:= \emptyset$.
For every finite subset $Z \subseteq P(T_{i,j})$ we consider a maximal family $\mathcal{F}_{i, j}(Z)$ of finite connected sets intersecting $P(T_{i, j})$ in exactly $Z$ such that every pair of distinct elements in $\mathcal{F}_{i, j}(Z)$ is disjoint outside $Z$.
We add $\bigcup \mathcal{F}_{i, j}(Z)$ to $Y_{i, j}$ if $|\mathcal{F}_{i, j}(Z)| \leq \kappa_i$.
By construction, $|Y_{i, j}| \leq \kappa_i$ holds.
We set $T(Y_{i,j}):= \{t \in T: P_t \cap Y_{i,j} \neq \emptyset \}$.

We ensure that the following conditions hold for all $i < cf(\kappa)$ and $j < \omega_1$
\begin{enumerate}[label=(\roman*)]
    \item\label{itm:decomp_a} $\Set{t_k: k < \kappa_i} \subseteq V(T_{i, 0})$,
    \item\label{itm:decomp_b} $\bigcup\Set{V(T_{i', j'}): i' \leq i, j' \leq j} \subseteq V(T_{i,j})$,
    \item\label{itm:decomp_c} $\Set{t_{i', j}^{k}: k < \kappa_i} \subseteq V(T_{i, j+ 1})$ for all $i < i' < \cf(\kappa)$, and
    \item\label{itm:decomp_d} $T(Y_{i,j'}) \subseteq V(T_{i,j})$ for $j' < j$.
\end{enumerate}
The conditions \labelcref{itm:decomp_a,itm:decomp_b,itm:decomp_c,itm:decomp_d} specify a set of $\kappa_i$-many elements, which have to be contained in $T_{i,j}$.
We apply \cref{thm:closure} to obtain a rooted subtree $T_{i,j} \subseteq T$ of size $\kappa_i$ containing this set such that every component in $\cK(S \setminus P(T_{i,j}))$ has finite adhesion to $P(T_{i,j})$.

Finally, we set $T_i := \bigcup_{j < \omega_1} T_{i,j}$ for every $i < \cf(\kappa)$ and show that the sequence $(T_i)_{i < \cf(\kappa)}$ is as desired.
Note that $|T_i| = \omega_1 \cdot \kappa_i = \kappa_i < \kappa$.
Thus, by \cref{cor:normal_partition_tree_countable_branch}, $|P(T_i)| < \kappa$ holds.
By \labelcref{itm:decomp_b}, the sequence $(T_i)_{i < \cf(\kappa)}$ is increasing.
Note that \labelcref{itm:decomp_c} ensures that $T_{\ell, j} \subseteq \bigcup_{i < \ell} T_{i, j + 1}$ for every limit $\ell < \cf(\kappa)$ and every $j< \omega_1$.
This implies that $(T_i)_{i < \cf(\kappa)}$ is continuous.
Since all components in $\cK(S \setminus P(T_{i,j}))$ have finite adhesion to $P(T_{i,j})$ for $j < \omega_1$ and by \cref{prop:finite_adhesion_uncountable}, all components in $\cK(S \setminus P(T_i))$ have finite adhesion to $P(T_i)$ for every $i < \kappa$.

It remains to prove that $P(T_i)$ is $\kappa_i$-strong for every $i < \cf(\kappa)$.
Let $i < \cf(\kappa)$ be arbitrary, let $C$ be some finite connected set and let $X \subseteq S \setminus P(T_i)$ be some set with $|X| \leq \kappa_i$.
We have to show that there exists a finite connected set $C' \subseteq S \setminus X$ with $C \cap P(T_i) = C' \cap P(T_i)$.
If $C \subseteq P(T_i)$ holds, the connected set $C':= C$ is as desired.
Thus we can assume that $C \setminus P(T_i) \neq \emptyset$ and set $Z:= C \cap P(T_i)$.

Let $j < \omega_1$ such that $Z \subseteq P(T_{i,j})$.
By maximality of $\mathcal{F}_{i, j}(Z)$, the set $C \setminus Z$ intersects $\bigcup \mathcal{F}_{i, j}(Z)$.
We did not add $\bigcup \mathcal{F}_{i, j}(Z)$ to $Y_{i,j}$ in the construction of $T_{i, j+1}$, as $C \setminus Z \subseteq S \setminus P(T_i) \subseteq S \setminus P(T_{i, j + 1})$.
Thus the family $\mathcal{F}_{i, j}(Z)$ has size bigger than $\kappa_i$.
As $|P(T_i)| = \kappa_i$ and $|X| \leq \kappa_i$, there is a connected set $C' \in \mathcal{F}_{i,j}(Z)$ such that $(C' \setminus Z) \cap P(T_i) = \emptyset$ and $C' \cap X = \emptyset$.
Thus $C'$ is as desired, which proves that every $P(T_i)$ is $\kappa_i$-strong.
This completes the proof.
\end{proof}

\section{Characterisation via countable separation number} \label{sec:characterisation_countable_separation_number}

\begin{prop}\label{prop:cofinal_finite_adhesion}
	Let $(S, \cC)$ be a connected connectoid and $\hat S \subseteq S$ a subset such that every component in $\cK(S \setminus \hat S)$ has finite adhesion to $\hat S$. Further, let $T$ be some normal tree containing $\hat S$ cofinally.
	Then every component in $\cK(S \setminus V(T))$ has finite adhesion to $V(T)$. 
\end{prop}
\begin{proof}
	Let $K$ be an arbitrary component in $\cK(S \setminus V(T))$.
	Let $K'$ be the component in  $\cK(S \setminus \hat S)$ that is superset of $K$.
	As $K'$ has finite adhesion to $\hat S$, there is a finite set $X \subseteq \hat S$ such that $K'$ is a component in $\cK(S \setminus X)$.
	We prove that $K$ is a component in $\cK(S \setminus  \Down{X}_T)$, which shows that $K$ has finite adhesion to $V(T)$.
	
	Suppose for a contradiction that $K$ is not a component in $\cK(S \setminus  \Down{X}_T)$.
	Then $K$ is a proper subset of a component $K'' \in \cK(S \setminus \Down{X}_T)$.
	Since $K \in \cK(S \setminus V(T))$, $K''$ contains some $t \in V(T) \setminus \Down{X}_T$.
	By \cref{prop:equivalence_weak_normal_tree}, $\Up{t}_T \subseteq K_t^T$ is contained in $K''$ since $K_t^T \cap \Down{X}_T \subseteq K_t^T \cap V(T) \cap \Down{X}_T = \Up{t}_T \cap \Down{X}_T = \emptyset$.
	Since $T$ contains $\hat S$ cofinally, $\hat S \cap \Up{t}_T \neq \emptyset$ and in particular $K'' \cap \hat S \neq \emptyset$.
	As $K' \in \cK(S \setminus X)$ is a superset of $K'' \in \cK(S \setminus  \Down{X}_T)$, also $K'$ contains an element of $\hat S$.
	This contradicts the fact that $K'$ is a component in $\cK(S \setminus \hat S)$ and completes the proof.
\end{proof}

We are now ready to prove:
\CountableSeparationNumber*
The proof of~\cref{thm:countable_separation_number} is inspired by the proof of Pitz' result~\cite{pitz2021proof}*{Theorem 1.1}.
\begin{proof}
For the forward direction we assume that $(S, \cC)$ has a normal spanning tree $T$.
By~\cref{prop:well-order_respecting}, there is well-order $\leq$ of $S$ respecting the tree order of $T$.
We show that $\leq$ witnesses countable separation number.
Let $s \in S$ be arbitrary.
Note that the finite set $\ODown{s}_T$ is contained in $\ODown{s}_\leq$ by the choice of $\leq$.
The component $K_s^T \in \cK(S \setminus \ODown{s}_T)$ containing $s$ equals $\Up{s}_T$ by \cref{prop:equivalence_weak_normal_tree} and since $T$ is normal.
By choice of $\leq$, $K_s^T = \Up{s}_T$ is contained in $\Up{s}_\leq$. Therefore $K_s^T$ avoids $\ODown{s}_\leq$, which completes the proof of the forward direction.

For the backward direction, we have to prove that every connected connectoid with countable separation number has a normal spanning tree.
We prove this statement by strong induction on the size of $S$. Let $(S, \cC)$ be a connected connectoid of size $\kappa$ and assume that the statement is true for every connected connectoid of size less than $\kappa$.

If $\kappa$ is countable, $S$ is a countable union $\bigcup_{s \in S} \{s\}$ of dispersed sets and by \cref{thm:jung} there exists a normal spanning tree of $(S, \cC)$.
Thus we can assume that $\kappa$ is uncountable.
By \cref{thm:decomp}, there is a continuous increasing sequence $(S_i)_{i < \cf(\kappa)}$ of subsets of $S$ with $S = \bigcup_{i < \cf(\kappa)} S_i$ such that for every $i < \cf(\kappa)$
\begin{itemize}
    \item $|S_i| < \kappa$,
    \item every component in $\cK(S \setminus S_i)$ has finite adhesion to $S_i$, and
    \item $S_i$ is $( < \omega)$-strong.
\end{itemize} 

First of all, we show that $S_i$ is a countable union of sets that are dispersed with respect to $(S, \cC)$ for every $i < \cf(\kappa)$.  The torso $(S_i, \cC \restriction_{S_i}):= (S_i, \{C \cap S_i: C \in \cC \})$ has countable separation number, by \cref{prop:torso}.
Thus by the induction hypothesis $(S_i, \cC \restriction_{S_i})$ has a normal spanning tree.
By \cref{thm:jung}, $S_i$ is a countable union $\bigcup_{n \in \NN} X_n$ of sets that are dispersed with respect to $(S_i, \cC \restriction_{S_i})$.
We show that $X_n$ is also dispersed with respect to $(S, \cC)$ for every $n \in \NN$, which proves that $S_i$ is a countable union of sets that are dispersed with respect to $(S, \cC)$.

Suppose for a contradiction that there exists a necklace $N$ in $(S, \cC)$ that has infinite intersection with $X_n$.
Using the technique of \cref{sec:subconnectoid_minor}, we can construct a witness $(H_k)_{k \in \NN}$ of $N$ such that $H_k \cap H_{k + 1}  \cap X_n \neq \emptyset$ for every $k \in \NN$.
We prove that $(H_k \cap S_i)_{k \in \NN}$ is a witness of a necklace in $(S_i, \cC \restriction_{S_i})$.
Since $X_n \subseteq S_i$, the set $H_k \cap {S_i}$ is nonempty for every $k \in \NN$ and $H_k \cap S_i \in \cC \restriction_{S_i}$ by the choice of $\cC \restriction_{S_i}$.
Further, $(H_k \cap {S_i}) \cap (H_\ell \cap {S_i}) \neq \emptyset$ holds if and only if $|k - \ell | \leq 1$ for every $k, \ell \in \NN$, by the choice of $(H_k)_{k \in \NN}$ and since $X_n \subseteq S_i$.
Thus $(H_k \cap S_i)_{k \in \NN}$ is a necklace in $(S_i, \cC \restriction_{S_i})$ with infinite intersection with $X_n$, contradicting the dispersedness of $X_n$ in $(S, \cC \restriction_{S_i})$.

Now we turn our attention to the construction of the desired normal spanning tree of $(S, \cC)$.
More precisely, we construct recursively an increasing sequence $(T_i)_{i < \cf(\kappa)}$ of rooted normal trees such that $T_i$ contains $S_i$ cofinally and every component in $\cK(S \setminus V(T_i))$ has finite adhesion to $V(T_i)$ for every $i < cf(\kappa)$.
Then $\bigcup_{i < \cf(\kappa)} T_i$ is a normal spanning tree of $(S, \cC)$.

Set $T_0 := \emptyset$.
We pick an arbitrary ordinal $1 \leq i < \cf(\kappa)$ and assume that $T_j$ has been constructed for every $j < i$.
If $i$ is a limit, set $T_i:= \bigcup_{j < i} T_j$. Note that $T_i$ contains $S_i$ cofinally since the sequence $(S_i)_{i < \cf(\kappa)}$ is continuous. By assumption, every component in $\cK(S \setminus S_i)$ has finite adhesion to $S_i$. Then \cref{prop:cofinal_finite_adhesion} ensures that every component in $\cK(S \setminus V(T_i))$ has finite adhesion to $V(T_i)$.

If $i$ is a successor, we use \cref{prop:extension_normal_tree} to construct the tree $T_i$:
Let $K$ be an arbitrary component in $\cK(S \setminus V(T_{i - 1}))$.
Since $S_i$ is a countable union of dispersed sets with respect to $(S, \cC)$, also $S_i \cap K$ is a countable union of dispersed sets with respect to the induced subconnectoid $(K, \cC^K)$ on $K$.
By \cref{thm:jung}, there is a normal tree $T_K$ in $(K, \cC^K)$ containing $S_i \cap K$.
By deleting all vertices $t \in T_K$ for which $\Up{t}_{T_K}$ avoids $S_i \cap K$, we can assume that $T_K$ contains $S_i \cap K$ cofinally.
Thus the normal tree $T_i$ obtained by applying \cref{prop:extension_normal_tree}  to $T_{i - 1}$ and $(T_K)_{K \in \cK(S \setminus V(T_{i - 1}))}$ contains $S_i$ cofinally.
By \cref{prop:cofinal_finite_adhesion}, every component in $\cK(S \setminus V(T_i))$ has finite adhesion to $V(T_i)$.
This completes the backward direction.
\end{proof}
We can deduce by \cref{prop:torso}:
\begin{cor}
	The existence of normal spanning trees is closed under taking connected torsos at $(< \omega)$-strong subsets.
\end{cor}

\section{Open problems} \label{sec:open_problems}
In this section we present open problems regarding the characterisation of normal spanning trees in connectoids, which are inspired by results for undirected graphs.

We begin with an open problem regarding the characterisation via countable colouring number.
The definition of countable colouring number \cite{pitz2021proof} can be generalised to connectoids in the following way:
A connectoid $(S, \cC)$ has \emph{countable colouring number} if there exists a well-order $\leq$ of $S$ such that for every $s \in S$ there is a finite set $X \subseteq \ODown{s}_\leq$ such that $\{s\}$ is a component in $\cK(S \setminus (X \cup \OUp{s}_\leq))$.

Pitz proved that a connected undirected graph $G$ has a normal spanning tree if and only if every minor of $G$ has countable colouring number \cite{pitz2021proof}*{Theorem~1.1}.
This statement cannot be transfered to connectoids verbatim using our definition of minor:
\begin{prop}\label{prop:counterexample_colouring_number}
	There exists a directed graph $D$ such that the corresponding connected connectoid $(V(D), \cC)$ does not have a normal spanning tree and every minor of $(V(D), \cC)$ has countable colouring number.
\end{prop}
\begin{proof}
	We begin by constructing the directed graph $D$ (see \cref{fig:ccn_and_minor_not}).
	We define $A:=\{a_i: i < \omega\}$, $B:=\{b_j : j < \omega_1\}$ and $C:= \{c_{j, k} : j < \omega_1, k < \omega  \}$.
	Let $D$ be the directed graph with $V(D):= A \cup B \cup C$ and
	$E(D) := \{a_ib_j: i < \omega, j < \omega_1\} \cup  \{b_jc_{j, k}: j < \omega_1, k < \omega\} \cup \{c_{j,k}a_i: i < \omega, j < \omega_1, k < \omega\}$.
	
	We show that $V(D)$ is not a countable union of dispersed sets.
	Then by \cref{thm:jung}, $(V(D), \cC)$ does not have a normal spanning tree.
	Suppose for a contradiction that $V(D)$ is a countable union of dispersed sets.
	Then there exists a dispersed set $X$ that contains infinitely many elements of $B$.
	Let $(j(n))_{n \in \NN}$ be a sequence of distinct ordinals such that $b_{j(n)} \in B \cap X$ for every $n \in \NN$.
	For every $n \in \NN$, let $H_n \in \cC$ be induced by the directed cycle on $a_n, b_{j(n)}, c_{j(n),0}, a_{n+1}, b_{j(n+1)}, c_{j(n+1),0}, a_n$.
	Then $(H_n)_{n \in \NN}$ is a necklace that contains infinitely many elements of $B \cap X$, contradicting the dispersedness of $X$.
	
	Let $(\cP, \cC')$ be an arbitrary minor of $(V(D), \cC)$.
	Note that $\cP$ is a partition of a subset of $V(D)$.
	We show that $(\cP, \cC')$ has countable colouring number.
	Let $\leq$ be some well-order of $\cP$ such that $\{P \in \cP: P \cap A \neq \emptyset \}$ and $\{P \in \cP: P \cap (A \cup B) \neq \emptyset \}$ form initial segments and $\{P \in \cP: P \cap A \neq \emptyset \}$ has order type $\omega$.
	
	Since every non-singleton element of $\cC$ contains an element of $A$, every element in $\cP \subseteq \cC$ either intersects $A$ or is a singleton.
	Every element of $\cP$ intersecting $A$ has finite $\leq$-down-closure, by the choice of $\leq$.
	Thus every such vertex satisfies the condition for countable colouring number.
	
	If $\{b_j\}\in \cP$ for some $j < \omega_1$, then the singletons $\{c_{j,k}\}$ for $k < \omega$ have to be elements of $\cP$ since every non-singleton element of $\cC$ containing some $c_{j,k}$ contains $b_j$.
	Furthermore, every $\{c_{j,k}\}$ for $k< \omega$ is contained in $\Up{\{b_j\}}_\leq$, by the choice of $\leq$.
	Since every non-singleton connected set in $\cC$ containing $b_j$ has to contain some $c_{j,k}$, every non-singleton connected set in $\cC'$ containing $\{b_j\}$ has to contain some $\{c_{j,k}\}$.
	Thus $\{b_j\}$ satisfies the condition for countable colouring number.
	
	Finally, since every non-singleton connected set in $\cC$ containing $c_{j,k}$ for some $j < \omega_1$ and some $k < \omega$ has to contain $b_j$, every non-singleton connected set in $\cC'$ containing $\{c_{j,k}\}$ has to contain the element $P \in \cP$ with $b_j \in P$.
	Thus every singleton $\{c_{j,k}\}$ satisfies the condition for countable colouring number.
	This completes the proof.
\end{proof}

\begin{center}
	\begin{figure}[ht]
		\begin{tikzpicture}

			\foreach \x in {0,3} \draw[fill] (\x,0) circle [radius=.05];
			
			\foreach \y in {0,0.5} \draw[fill] (6,\y) circle [radius=.05];
			
			\draw[fill, opacity=0.4, gray] (6,1) circle [radius=.05];
			
			\draw[edge] (0,0) -- (2.95,0);
			\draw[edge] (3,0) -- (5.95,0);
			\draw[edge] (3,0) -- (5.95,0.5);
			\draw[edge, path fading=east] (3,0) -- (5.95,1);
			
			\draw[edge] (6,0) to [out=160,in=20] (0.03,0.03);
			\draw[edge] (6,0.5) to [out=160,in=20] (0.03,0.03);
			\draw[edge, path fading=east] (6,1) to [out=160,in=20] (0.03,0.03);
			
			\draw (0,-0.3) node {$a_i$};
			\draw (3,-0.3) node {$b_j$};
			\draw (6.5,0) node {$c_{j,0}$};
			\draw (6.5,0.5) node {$c_{j,1}$};
			\draw[opacity=0.4] (6.5,1) node {$c_{j,2}$};
		\end{tikzpicture}
	\caption{The schema of the directed graph $D$ in \cref{prop:counterexample_colouring_number} for some $i < \omega$ and $j<\omega_1$.}
\label{fig:ccn_and_minor_not}
	\end{figure}
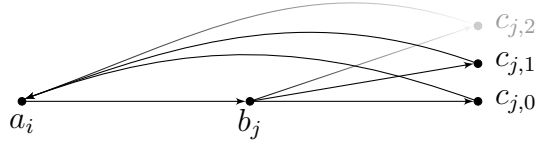
\end{center}

\noindent
\cref{prop:counterexample_colouring_number} suggests that we should weaken the definition of minor if we want to find a characterisation via countable colouring number.
But we do not want to weaken the notion of minor too much:
We still want the property of having a normal spanning tree to be preserved under taking minors.

The following type of minor seems to be a good candidate:
Given a connectoid $(S, \cC)$, some partition $\mathcal{P}$ of $S$ and a family $(c_P)_{P \in \mathcal{P}}$ of elements $c_P \in P$, let $\mathcal{F}^\mathcal{P}$ be the set of finite sets $X \subseteq \mathcal{P}$ for which there is $C \in \cC$ with $\{c_P: P \in X \} \subseteq C \subseteq \bigcup X$.
Note that $\mathcal{F}^\mathcal{P}$ satisfies~\labelcref{itm:bonding_closed,itm:contains_singletons}.
Let $(\mathcal{P}, \cC^\mathcal{P})$ be the connectoid induced by $\mathcal{F}^\mathcal{P}$.
We say $(\mathcal{P}, \cC^\mathcal{P})$ is obtained by \emph{weak contraction} of $(S, \cC)$.
A connectoid $(S'', \cC'')$ is a \emph{weak minor} of $(S, \cC)$ if there exists a subconnectoid $(S', \cC')$ of $(S, \cC)$ such that $(S'', \cC'')$ is obtained by weak contraction of $(S', \cC')$.\footnote{In undirected graphs, every graph-theoretic minor is a weak minor but weak minors are in general not graph-theoretic minors.}

Note that every necklace in a weak minor induces a necklace in the host connectoid.
Thus weak minors inherit the property of having a normal spanning trees by \cref{thm:jung}.
Further, given a normal spanning tree $T$ of a connectoid $(S, \cC)$, by~\cref{prop:well-order_respecting} there is a well-order $\leq$ of $S$ respecting the tree order of $T$.
Then $\leq$ witnesses countable colouring number of $(S, \cC)$.
This implies that every weak minor of a connectoid with a normal spanning tree has countable colouring number.

\begin{problem}
	Let $(S, \cC)$ be a connected connectoid.
	If every weak minor of $(S, \cC)$ has countable colouring number, does there exist a normal spanning tree of $(S, \cC)$?
\end{problem}

We note that the connectoid in the proof of~\cref{prop:counterexample_colouring_number} contains a weak minor without normal spanning tree:
Let $\mathcal{P}$ be the partition consisting of $A_i:=\{a_i\}$ for $i < \omega$ and $B_j:=\{b_j\} \cup \{c_{j,k} :k < \omega\}$ for $j < \omega_1$.
Furthermore, let $c_{A_i}:= a_i$ and $c_{B_j}:= b_j$.
We set $\mathcal A:= \{A_i: i < \omega\}$ and $\mathcal B:= \{B_j: j < \omega_1\}$.
Note that for every $A_i \in \mathcal A$ and every $B_j \in \mathcal B$ there is a connected set in $\cC$ that is subset of $A_i \cup B_j$ and contains $a_i$ and $b_j$.
Thus the set $\mathcal{F}^\mathcal{P}$ contains $\{A_i,B_j\}$ for every $A_i \in \mathcal A$ and every $B_j \in \mathcal B$.
Given an arbitrary well-order $\leq$ of $\mathcal{P}$, there is $P \in \mathcal{P}$ such that either $P \in \mathcal A$ and $\ODown{P}_\leq$ contains infinitely many elements of $\mathcal B$ or $P \in \mathcal B$ and $\ODown{P}_\leq$ contains infinitely many elements of $\mathcal A$.
In both cases, there is no finite set $X \subseteq \ODown{P}_\leq$ such that $\{P\}$ is a component in $\cK(S \setminus (X \cup \OUp{P}_\leq))$.

Pitz proved \cite{pitz2021proof}*{Theorem~1.2} a forbidden minor characterisation of the existence of normal spanning trees in undirected graphs by combining his characterisation of normal spanning trees via countable colouring number with the first author, Carmesin, Komj\'{a}th and Reiher's~\cite{bowler2019colouring} characterisation of countable colouring number in undirected graphs.

\begin{problem}
	Does there exist a forbidden (weak) minor characterisation of the existence of normal spanning trees?
\end{problem}

\section*{Acknowledgement}

The second author gratefully acknowledges support by a doctoral scholarship of the Studienstiftung des deutschen Volkes.

\bibliography{ref.bib}

\end{document}